\documentclass[oneside,12pt]{article}

\usepackage{amsmath}
\usepackage{amssymb}
\usepackage{eucal}
\usepackage{mathrsfs}
\usepackage{theorem}
\usepackage{color}
\usepackage{enumitem}
\usepackage{hyperref}
\usepackage[normalem]{ulem}
\usepackage[margin=3cm]{geometry}

\usepackage{array}

\newcolumntype{L}[1]{>{\raggedright\let\newline\\\arraybackslash\hspace{0pt}}m{#1}}
\newcolumntype{C}[1]{>{\centering\let\newline\\\arraybackslash\hspace{0pt}}m{#1}}
\newcolumntype{R}[1]{>{\raggedleft\let\newline\\\arraybackslash\hspace{0pt}}m{#1}}

\usepackage[all]{xy}

\theoremheaderfont{\bfseries}
\newtheorem{theorem}{Theorem}[section]
\newtheorem{lemma}[theorem]{Lemma}
\newtheorem{proposition}[theorem]{Proposition}

\newtheorem{definition}[theorem]{Definition}

{\theorembodyfont{\rmfamily}
\newenvironment{proof}{{\flushleft \emph{Proof}:}}{\hfill$\blacksquare$}

\newcommand{\g}{\frakg}
\newcommand{\h}{\frakh}

\newcommand{\X}{\frakX}
\newcommand{\hn}{\frakn}
\newcommand{\euc}{\frake}
\newcommand{\R}{\bbR}
\newcommand{\M}{\calM}
\newcommand{\N}{\calN}
\newcommand{\E}{\calE}
\newcommand{\Vol}{\text{Vol}}
\newcommand{\dVol}{\textup{d}\text{Vol}}
\newcommand{\Volume}{\textup{d}\text{Vol}_\g}

\newcommand{\Emph}[1]{{\bfseries #1}}
\newcommand{\tM}{\tilde{\calM}}

\newcommand{\tE}{\tilde{\E}}
\newcommand{\vp}{\varphi}
\newcommand{\e}{\varepsilon}

\newcommand{\fTN}{f^*T\N}

\newcommand{\nabfTN}{\nabla^{\fTN}}

\newcommand{\Imm}{\operatorname{Imm}}

\newcommand{\wc}{\rightharpoonup}

\newcommand{\cofop}{\operatorname{cof}}
\renewcommand{\O}{\operatorname{O}}
\newcommand{\SO}{\operatorname{SO}}
\newcommand{\II}{\operatorname{II}}
\newcommand{\rank}{\operatorname{rank}}
\newcommand{\dist}{\operatorname{dist}}
\newcommand{\Hom}{\operatorname{Hom}}
\newcommand{\id}{{\operatorname{Id}}}
\newcommand{\image}{{\operatorname{Image}}}

\newcommand{\ip}[1]{\langle #1 \rangle}

\newcommand{\txi}{\tilde{\xi}}

\def\Xint#1{\mathchoice
{\XXint\displaystyle\textstyle{#1}}%
{\XXint\textstyle\scriptstyle{#1}}%
{\XXint\scriptstyle\scriptscriptstyle{#1}}%
{\XXint\scriptscriptstyle\scriptscriptstyle{#1}}%
\!\int}
\def\XXint#1#2#3{{\setbox0=\hbox{$#1{#2#3}{\int}$ }
\vcenter{\hbox{$#2#3$ }}\kern-.6\wd0}}

\newcommand{\dashint}{\Xint-}

\newcommand{\textand}{\quad\text{ and }\quad}
\newcommand{\Textand}{\qquad\text{ and }\qquad}

\newcommand{\brk}[1]{\left(#1\right)}          
\newcommand{\BRK}[1]{\left\{#1\right\}}        

\newcommand{\beq}{\begin{equation}}
\newcommand{\eeq}{\end{equation}}

\newcommand{\limn}{\lim_{n\to\infty}}

\newcommand{\secref}[1]{Section~\ref{#1}}

\newcommand{\thmref}[1]{Theorem~\ref{#1}}

\newcommand{\propref}[1]{Proposition~\ref{#1}}

\newcommand{\calE}{{\mathcal E}}

\newcommand{\calM}{{\mathcal M}}
\newcommand{\calN}{{\mathcal N}}

\newcommand{\calU}{{\mathcal U}}
\newcommand{\calV}{{\mathcal V}}

\newcommand{\frakX}{\mathfrak{X}}

\newcommand{\frake}{\mathfrak{e}}

\newcommand{\frakg}{\mathfrak{g}}
\newcommand{\frakh}{\mathfrak{h}}

\newcommand{\frakn}{\mathfrak{n}}

\newcommand{\frakx}{\mathfrak{x}}

\newcommand{\bbR}{{\mathbb R}}

\numberwithin{equation}{section}

\begin{document}

\title{Stability of isometric immersions of hypersurfaces}

\author{Itai Alpern\footnote{Einstein Institute of Mathematics, The Hebrew University, Jerusalem, Israel
}
\and Raz Kupferman\footnotemark[1] 
\and   Cy Maor\footnotemark[1] 
}
\date{}
\maketitle

\begin{abstract}
We prove a stability result of isometric immersions of hypersurfaces in Riemannian manifolds, with respect to $L^p$-perturbations of their  fundamental forms: 
For a manifold $\M^d$ endowed with a reference metric and a reference shape operator, we show that a sequence of immersions $f_n:\M^d\to\N^{d+1}$, whose pullback metrics and shape operators are arbitrary close in $L^p$ to the reference ones, converge to an isometric immersion having the reference shape operator.
This result is motivated by elasticity theory and generalizes a previous result \cite{AKM22} to a general target manifold $\N$, removing a constant curvature assumption.
The method of proof differs from that in \cite{AKM22}: it extends a Young measure approach that was used in codimension-0 stability results, together with an appropriate relaxation of the energy and a regularity result for immersions satisfying given fundamental forms.
In addition, we prove a related quantitative (rather than asymptotic) stability result in the case of Euclidean target, similar to \cite{CMM19b} but with no a-priori assumed bounds.
\end{abstract}

{\bfseries Keywords}: Rigidity,
Riemannian manifolds, Isometric immersions, Non-Euclidean elasticity

{\footnotesize{
\tableofcontents
}}

\section{Introduction}

\paragraph{Background: rigidity in codimension 0}
In 1850, Liouville proved that isometries in $\R^d$ are rigid \cite{Lio50}: if $\Omega \subset \R^d$ is open and connected and $f\in C^1(\Omega ; \R^d)$, and if for every $x\in \Omega$, $df(x)\in \SO(d)$, then, in fact, $f$ is a rigid motion, $f(x) = Qx+b$ for some $Q\in \SO(d)$ and $b\in \R^d$.
This theorem has been generalized over the years in a number of ways:
\begin{itemize}
\item \Emph{Regularity}: Reshetnyak showed that the theorem holds for Lipschitz maps \cite{Res67b}. 
This is essentially a regularity result, as the proof proceeds by showing that such Lipschitz maps are harmonic, hence smooth.
(Reshetnyak's theorem can be reformulated as follows: if a Lipschitz map is a.e. orientation-preserving and pulls backs a smooth metric, then it is smooth.)
\item \Emph{Asymptotic stability}: Reshetnyak also showed that this rigidity is stable: if $\dist(df_n, \SO(d)) \to 0$ in $L^p$ for some $p\in [1,\infty)$, then, modulo a subsequence and translations, $f_n\to Qx+b$ in $W^{1,p}(\Omega;\R^d)$ \cite{Res67}.
\item \Emph{Quantitative stability}: Friesecke, James and M\"uller (FJM) (\cite{FJM02b}, see also \cite[Section~2.4]{CS06}) proved a quantitative version of this theorem: for $p\in (1,\infty)$, there exists $C=C(\Omega,p)$ such that there exists for every $f\in W^{1,p}(\Omega;\R^d)$ a rigid map $\bar{f}(x) = Qx+b$, such that
\[
\|f-\bar{f}\|_{W^{1,p}(\Omega;\R^d)} \le C \|\dist(df, \SO(d))\|_{L^p(\Omega)}.
\]
\end{itemize}
There are many generalizations in other directions (e.g., for conformal maps, multiple energy wells, incompatible fields), which are beyond the scope of this short discussion.
In the context of Riemannian geometry, Liouville's and Reshetnyak's theorems generalize to maps between two compact oriented Riemannian manifolds $(\M,\g)$ and $(\N,\h)$ of the same dimension. In this context we denote by $\SO(\g_q,\h_{q'})$ the space of orientation-preserving isometries $T_q\M\to T_{q'}\N$ (or simply by $\SO(\g,\h)$, when no confusion may arise).
\begin{itemize}
\item \Emph{Rigidity}: The following well-known fact can be thought of as a generalization of Liouville's theorem: if $f\in C^2(\M;\N)$ satisfies $df_q \in \SO(\g_q,\h_{f(q)})$ for every $q\in \M$, then $f$ is rigid, in the sense that it is determined by its value and the value of its differential at a single point:
Every $q\in\M$ has an open neighborhood $U\ni q$, such that 
\[
f|_U = \exp^\h_{f(q)} \circ df_q \circ (\exp^\g_q)^{-1},
\]
where $\exp^\g$ and $\exp^\h$ are the exponential maps in $\M$ and $\N$.
\item \Emph{Regularity}: If $f:\M\to \N$ is Lipschitz and satisfies $df \in \SO(\g,\h)$ almost everywhere, then $f$ is a smooth isometric immersion and the previous statement applies.
This result (or variants of it) has been proved several times, using various techniques \cite{Res78,Res94,LP11,LS14,KMS19}.
\item \Emph{Asymptotic stability}: 
It was shown in \cite{KMS19}  that if $f_n\in W^{1,p}(\M;\N)$ satisfy $\dist(df_n,\SO(\g,\h))\to 0$ in $L^p(\M)$, then (modulo a subsequence) $f_n$ converges strongly in $W^{1,p}$ to a smooth isometric immersion $f:\M\to \N$ (in particular, such  an isometric immersion exists).
\end{itemize}
A generalization of the FJM quantitative stability result to a Riemannian setting is yet missing, and even its formulation as a conjecture is not obvious.
A particular case concerning mappings between round spheres appear in \cite[Theorem~3.2]{CLS22}; a stronger version of that result can in fact be deduced from the Euclidean result and will be proved in a forthcoming paper \cite{KM24}.

The above results have direct relevance to elasticity theory: 
In the Euclidean settings $\|\dist(df, \SO(d))\|_{L^p}^p$ is a prototypical model for the elastic energy of a strained elastic solid, hence these theorems relate the smallness of the elastic energy to being close to a zero energy state; essentially, any rigorous derivation of a low-energy limit in elasticity uses these results.
The Riemannian case arises in non-Euclidean elasticity, which is an elastic theory for pre-stressed bodies (see, e.g., \cite{KES07,ESK09a,ESK13,LM22}).  
The asymptotic stability result in \cite{KMS19}, for example, implies that (as expected by physicists) if $\M$ cannot be immersed isometrically in $\N$,
then the infimal elastic energy is positive, and the elastic body cannot release all of it stresses by, say, forming microstructures. 

\paragraph{Main result: asymptotic stability of isometric immersions in codimension-1}
In this paper we are concerned with codimension-1 versions of the above results.
The classical theorem in this context (for a Euclidean target) is the fundamental theorem of surface theory (see, e.g., \cite{Ten71}):
\begin{quote}
Let $(\M,\g)$ be an oriented, connected, simply-connected, compact $d$-dimensional Riemannian manifolds with Lipschitz boundary. 
Let $S$ by a smooth symmetric $(1,1)$ tensor field on $\M$. 
If $\g$ and $S$ satisfy the Gauss-Codazzi compatibility conditions, then there exists a smooth isometric immersion $f:\M\to\R^{d+1}$ having shape operator $S$. 
This immersion is unique modulo a composition with a Euclidean rigid map.
\end{quote}

We recall that for an immersion $f:\M\to \R^{d+1}$, we can define the \emph{Gauss map} $\hn_f:\M \to T\R^{d+1}\simeq \R^{d+1}$, where $\hn_f(q)$ is the unique unit vector such that $(df_q(u_1),\ldots,df_q(u_d),\hn_f(q))$ is an oriented basis of $\R^{d+1}$ for any oriented basis $(u_1,\ldots,u_d)$ of $T_q(\M)$.
The \emph{second fundamental form} of an immersion is a $(2,0)$ tensor which is given by $\II_f = -df^{T}\nabla \hn_f : T\M\times T\M \to \R$; it contains the same information as the \emph{shape operator} $S_f:T\M\to T\M$, a (1,1) tensor implicitly defined via $\nabla \hn_f = -df \circ S_f$ (which is well defined since the image of $\nabla \hn_f$ is normal to $\hn_f$, and thus in the image of $df$).
Thus, given a metric $\g$ and a (1,1) tensor $S$, $f:\M\to\R^{d+1}$ is an isometric immersion having a shape operator $S$ if
\beq\label{eq:isometric_immersion_Euc}
df\in \O(\g,\euc) \textand \nabla\hn_f = -df\circ S,
\eeq
where $\euc$ is the Euclidean metric, and $\O(\g_q,\euc)$ is the set of isometric linear maps $T_q\M \to \R^{d+1}$.

The last clause of the fundamental theorem of surface theory is, again, a rigidity statement: An immersion $f:\M\to\R^{d+1}$ satisfying \eqref{eq:isometric_immersion_Euc} is determined by the values of $f(q)$ and $df_q$ at any $q\in \M$.
This rigidity statement generalizes to any target $(d+1)$-dimensional manifold $(\N,\h)$, by replacing in \eqref{eq:isometric_immersion_Euc} $\euc$ with $\h$ and $\nabla$ with $\nabla^{f^*T\N}$, the pullback connection of the Levi-Civita connection of $(\N,\h)$, which is the natural derivative of vector fields in $\N$ along $f$, i.e., maps that take $p\in \M$ to a vector in $T_{f(p)}\N$.\footnote{The uniqueness of an isometric immersion $f:\M\to \N$ having a given shape operator $S$ and given initial data $f(q)$ and $df_q$, follows from integrating along geodesics of $\M$: Given $q'\in \M$, let $\gamma$ be a geodesic from $q$ to $q'$. Choose a parallel orthonormal frame $(u_1,\ldots,u_d)$ along $\gamma$, with $u_1 = \dot \gamma$. 
Then the image of $\gamma$ and the frame under $f$ is obtained uniquely by solving the system $\frac{D}{dt}e_i = S_i^1 \hn$, $\frac{D}{dt}\hn = -S_1^i e_i$, where $e_i$ is the image of $u_i$, $\hn$ is their normal, and $\frac{D}{dt}$ is the covariant derivative in $(\N,\h)$. }

The main focus of this article is the asymptotic stability of isometric immersions.
The codimension-0 property of being $L^p$-close to an orientation-preserving isometry is replaced by being $L^p$-close to first and second fundamental forms. 
For $f:\M\to \N$ we define
\beq
\E_p(f) = 
\|\dist^p(df,\O(\g,\h))\|_{L^p(\M)}^p + \|\nabla^{f^*T\N} \hn_f + df\circ S\|_{L^p(\M)}^p.
\label{eq:Etotal}
\eeq 
The first term, denoted $\E_p^S(f)$, is called a \Emph{stretching energy} and the second term, denoted $\E_p^B(f)$, is called a \Emph{bending energy}.
As these name suggest, this energy is motivated by the elastic theory of thin bodies.
There, one often compares the second fundamental forms ($(2,0)$ tensors) rather than the shape operators ($(1,1)$ tensors); from a physical point of view these energies are equivalent, and using the shape operator is more convenient from a calculus of variations point of view. 
See \cite{AKM22} for further details and implications to elasticity theory.

The natural space on which the energy \eqref{eq:Etotal} is defined and finite is the space of \Emph{$p$-Sobolev immersions},
\beq
\label{eq:imm_p}
\Imm_p(\M;\N) = \{f\in W^{1,p}(\M;\N) ~:~ \text{$\rank df = d$ a.e. and $\hn_f\in W^{1,p}(\M;T\N)$}\},
\eeq
where $\hn_f$ is the unit normal vector field in $\N$ along $f$, which is defined a.e.
Note that for such non-smooth maps, the pull-back bundle $f^*T\N$ is a non-smooth bundle; we therefore define $\nabla^{f^*T\N} \hn_f$ using a connector operator, as explained in the next section.

Our main result is the following:
\begin{theorem}
\label{thm:main}
Suppose that there exists a sequence of $p$-immersions $f_n \in \Imm_p(\M;\N)$, $p\ge 1$, satisfying 
\[
\limn \E_p(f_n) \to 0.
\] 
Then there exists a subsequence of $f_n$ converging in $W^{1,p}(\M;\N)$ to a smooth isometric immersion $f:(\M,\g) \to (\N,\h)$. 
Furthermore, $\hn_{f_n} \to \hn_f$ in $W^{1,p}(\M;T\N)$, and the shape operator of the limit equals the reference shape operator, $\nabfTN\hn_f = - df\circ S$.
\end{theorem}

This result was obtained in \cite{AKM22} for the case where $(\N,\h)$ has constant sectional curvature.
The proof there was essentially based on reducing the problem to the codimension-0 problem, by ``thickening" $(\M,\g)$ into a $(d+1)$-dimensional manifold, $(\M\times[-h,h],G)$ for some $h>0$, extending $f$ into a map $F: \M\times[-h,h]\to \N$, and using the non-Euclidean version \cite{KMS19} of Reshetnyak's stability theorem in codimension-0.
This approach has difficulties generalizing beyond constant sectional curvature, since we exploited in constant curvature the fact that the metric $G$ is uniquely determined by $\g$ and $S$, independently of the maps $f_n$. 

We overcome this by using a different approach. 
As a first step, we need the following regularity theorem:
\begin{theorem}
\label{thm:reg}
Let $f\in \Imm_\infty(\M;\N)$
such that the first fundamental form $f^*\h$ and the shape operator $S_f$, implicitly defined by\footnote{
As in the Euclidean target case, the shape operator $S_f:T\M\to T\M$, is well defined by the definition of $\hn_f$ as a unit vector normal to the image of $df$. 
Note that linear maps in $T\M$ are canonically identified with $(1,1)$ tensors, i.e., as sections of $T^*\M\otimes T\M$.
}
\beq\label{eq:S_f}
\nabla^{f^*T\N} \hn_f = -df\circ S_f
\eeq
are smooth.
Then, $f$ is smooth.
\end{theorem}

Note that if $f$ was in $W^{2,p}$, this result would follow by bootstrapping the expressions for the second derivatives in terms the frame induced by $df$ and $\hn_f$.
However, we only know that $f\in W^{1,\infty}$ and that $\hn_f\in W^{1,\infty}$.
The improved regularity is then deduced from compactness of low energy configurations of thin bodies of arbitrary dimension and codimension \cite{KS14}, which, in turn, uses the FJM Euclidean rigidity estimate.
It is interesting whether one can deduce Theorem~\ref{thm:reg} directly from elliptic regularity, as done in codimension-0, using an adaptation of the Piola identity to codimension-1 (see \cite{KMS19}).

Given this regularity result,  it is sufficient, in order to prove Theorem~\ref{thm:main}, to prove that if $\E_p(f_n) \to 0$, then $f_n$ converges in $W^{1.p}$ to a zero energy map.
To this end, we deploy an approach using Young measures, first suggested for the proof of Reshetnyak's asymptotic stability theorem in \cite{JK89}, and then developed in \cite{KMS19} for proving the non-Euclidean equivalent result.
Since in our case, as in \cite{KMS19}, the target space of $f_n$ is not a vector space, one needs to work in local coordinates in order to use the Young measure limit. 

As such, this approach works for $p>d$, since only then maps in $W^{1,p}(\M;\N)$ are guaranteed to be localizable.
In \cite{KMS19}, the lower integrability regime $p\le d$ (which includes the most important case from the elasticity point of view, $p=2$), is obtained by first approximating $f_n$ by uniformly Lipschitz maps $\tilde{f}_n$.
A similar use of Lipschitz truncation for overcoming the localization problem of Sobolev maps between manifolds was done also in \cite{KM21}.

This approach, however, does not work directly on the energy $\E_p$, since it is not clear how to approximate immersions in $\Imm_p(\M;\N)$ by Lipschitz immersions (the immersion property is lost in the truncation process).
To overcome this difficulty, we first relax the energy $\E_p$ into an energy $\tE_p: W^{1,p}(\M;T\N) \to [0,\infty)$ (see Definition~\ref{def:relax_en}), which is defined and finite over all vector fields in $W^{1,p}(\M;T\N)$, and not only for vector fields that are perpendicular to their projection on $\N$, as in $\Imm_p(\M;\N)$.
From an elasticity point of view, the energy $\tE_p$ is essentially an energy for director fields, rather than configurations of the elastic body.
Thus, we prove a stronger version of Theorem~\ref{thm:main} for the energy $\tE_p$ (Theorem~\ref{thm:main2}) by applying Young measures for the case $p>d$, and precede this analysis with a Lipschitz truncation for $p\le d$, to get to the localizable regime.

\paragraph{Additional quantitative stability result:}
In the appendix, we include a short quantitative stability result, in the spirit of the FJM estimate:

We show in Theorem~\ref{thm:FJM_codim1} that if $\N = \R^{d+1}$, $p\in (1,\infty)$, and if the metric $\g$ of $\M$ and shape operator $S$ are compatible (i.e., if there exists an isometric immersion $\M\to \R^{d+1}$ with shape operator $S$), then for any $f\in \Imm_p(\M;\R^{d+1})$ there exists an isometric immersion $f_0:\M\to \R^{d+1}$ with shape operator $S$ such that
\[
\|f-f_0\|_{W^{1,p}}^p + \|\hn_f - \hn_{f_0}\|_{W^{1,p}}^p \le C \E_p(f),
\]
where $C>0$ is a constant independent of $f$.
In particular this implies Theorem~\ref{thm:main} for a Euclidean target space, under the assumption of compatibility.
Similar results, comparing the left-hand side of two immersions with an energy that measures stretching (discrepancy of the first fundamental forms) and bending (discrepancy of the second fundamental form), appear in \cite{CMM19b} (in particular Theorem~4.3).
These results, however, require $f$ to be in $\Imm_q(\M;\R^{d+1})$ for $q\ge \min\{p,2\}$, and, more importantly, requires a-priori uniform bounds on the first an second fundamental forms; Theorem~\ref{thm:FJM_codim1} does not require such a priori bounds, and also admits a simpler proof.
Additionally, Theorem~\ref{thm:FJM_codim1} can be generalized to director fields, using the relaxed energy $\tE_p$ (Definition~\ref{def:relax_en}), whereas the result of \cite{CMM19b} cannot.
The difference is mainly because of the different structures of the stretching and bending energy (in \cite{CMM19b}, the bending energy measures a discrepancy in the second fundamental forms rather than in the shape operators).
This can be seen as another evidence that the energy $\E_p$ is the natural ``stretching plus bending" energy from a calculus of variations point of view.

\begin{figure}
\label{fig:table}
\begin{center}
\begin{tabular}{ | L{3cm} | L{2cm} | L{2.7cm} | L{3cm} | L{2.5cm} |}
\hline
& \multicolumn{2}{ c |}{Codimension-0} 
& \multicolumn{2}{ c |}{Codimension-1} 
\\
\cline{2-5}
& Euclidean & Riemannian & Euclidean & Riemannian \\
\hline
Rigidity & \cite{Lio50} & ``Folklore" & Fundamental theorem of surfaces \cite{Ten71} & ``Folklore" \\
\hline
Regularity & \cite{Res67b} & \cite{Res78,Res94,LS14,KMS19} & \cite{AKM22} & Theorem~\ref{thm:reg} \\
\hline
$L^p$ Asymptotic Stability & \cite{Res67} &  \cite{KMS19} & \cite{AKM22} & Theorem~\ref{thm:main} \\
\hline
$L^p$ Quantitative Stability & \cite{FJM02b} & Constant curvature $\kappa>0$  \cite{CLS22,KM24} & 
\cite{CMM19b}, Theorem~\ref{thm:FJM_codim1}  & \\
\hline
\end{tabular}
\end{center}
\caption{A summary of the results presented in the introduction.
This list is not comprehensive: there are many other results on regularity of isometries (e.g., \cite{Har58,CH70,Tay06}), of isometric immersions without assumptions on the second fundamental form (e.g., \cite{Pak04,MP05,Hor11,LP13,JP17,HV18}), and of stability of immersions in Euclidean setting in various topologies, stronger than the $L^p$ stability discussed here \cite{Cia03,CM16,CM19,CMM20}.}
\label{fig:1}
\end{figure}

\paragraph{Structure of the paper}
In Section~\ref{sec:imm_p}, we review the definitions of Sobolev spaces between manifolds and prove some basic properties, in particular for Sobolev vector fields. We also introduce coordinate representations of the energy and the main fields in Section~\ref{sec:coordinates}.
In Section~\ref{sec:main_pf}, we define the relaxed energy (Section~\ref{sec:relaxed}) and prove Theorem~\ref{thm:main}, modulo the smoothness result (Theorem~\ref{thm:reg}), which we proof in Section~\ref{sec:smoothness}.
Finally, we prove in Appendix~\ref{sec:FJM_codim1} the codimension-1 Euclidean stability result. 
Figure~\ref{fig:1} displays a table in which our results are put in context with the existing literature.

\paragraph{Notations} 
Let $(V,\g)$ and $(W,\h)$ be oriented inner-product spaces. We denote the inner-products by $\ip{\cdot,\cdot}_\g$ and $\ip{\cdot,\cdot}_\h$. 
We  denote the corresponding norms by $|\cdot|$, unless the notations $|\cdot|_\g$ and $|\cdot|_\h$ help readability.  
We denote by $\Hom(V,W)$ the space of linear maps from $V$ to $W$; we denote by $\O(\g,\h)\subset\Hom(V,W)$ the subset of orthogonal maps; if $V$ and $W$ have equal dimensions, we denote by $\SO(\g,\h)\subset \O(\g,\h)$ the set of orientation-preserving orthogonal maps. 
We denote the norm on $\Hom(V,W)$ induced by $\g$ and $\h$ by $|\cdot|$, rather than by the more cumbersome notation $|\cdot|_{\g,\h}$. For a set $K\subset \Hom(V,W)$ and $A\in\Hom(V,W)$, we denote by $\dist(A,K)$ the distance (with respect to the norm $|\cdot|$) between the element $A$ and the set $K$. 
These notation carry on naturally if $V$ and $W$ are replaced by vector bundles over a manifold (in particular to the tangent bundle endowed with a Riemannian metric).

For a manifold $\M$, we denote by $\Omega^k(\M)$ the space of $k$-forms on $\M$. For a vector bundle $E\to \M$, we denote by $\Gamma(E)$ the space of sections of $E$, and by $\Omega^k(\M;E)$ the space of $E$-valued $k$-forms on $\M$. Likewise, we denote by $L^p\Gamma(E)$, $L^p\Omega^k(\M;E)$, $W^{1,p}\Gamma(E)$ and $W^{1,p}\Omega^k(\M;E)$ the corresponding spaces having $L^p$- and $W^{1,p}$-regularity; the corresponding norms are induced by norms on $T\M$ and $E$. Finally, for a Riemannian manifold $(\M,\g)$, $\Volume$ denotes the Riemannian volume form.
More notations are introduced in the next section in the context of mappings between manifolds.

\section{Sobolev maps between manifolds}
\label{sec:imm_p}

This work is concerned with Sobolev maps between Riemannian manifolds. 
Let $(\M,\g)$ and $(\X,\frakx)$ be Riemannian manifolds of arbitrary dimensions, with $\M$ compact, possibly with boundary, and $\X$ without boundary (in the sequel, $(\X,\frakx)$ is either a compact Riemannian manifold $(\N,\h)$, or its non-compact tangent space $(T\N,S^\h)$, where the metric $S^\h$ is described below). 
For smooth $f:\M\to \X$, we denote by $Df:T\M\to T\X$ the tangent map, acting as
\[
Df : (q,\xi) \mapsto (f(q), df_q(\xi))
\qquad q\in\M, \,\,\, \xi\in T_q\M,
\]
which is a linear bundle morphism covering $f$ (e.g., \cite[Chap.~2]{Sau89}).
The map $Df$ should be distinguished from the differential $df\in\Gamma(T^*\M\otimes f^*T\X)$, acting as
\[
df_q : \xi \mapsto df_q(\xi),
\] 
which  is a linear bundle map; $df$ is the pullback of $Df$, where as common, $f^*T\X$ denotes a vector bundle over $\M$, with the canonical identification $(f^*T\X)_q \simeq T_{f(q)}\X$. 

For $p\ge1$, the space of Sobolev maps $W^{1,p}(\M;\X)$ can be defined intrinsically, along with a notion of a weak derivative (see Convent and van Schaftingen \cite{CVS16} for a recent account). 
An equivalent extrinsic (and more common) definition is given by using Nash's embedding theorem, introducing an isometric embedding $\iota:(\X,\frakx)\to(\R^D,\euc)$, for some $D$ large enough. Then,
\[
W^{1,p}(\M;\X) = \{f:\M\to \X ~:~ \iota\circ f \in W^{1,p}(\M;\R^D)\}.
\]
This space inherits the strong and weak topologies of $W^{1,p}(\M;\R^D)$, and it is independent of the embedding $\iota$.
Moreover, $\iota\circ W^{1,p}(\M;\X)$ is a weakly closed subset of $W^{1,p}(\M;\R^D)$, and in particular bounded sequences have weakly-convergent subsequences for $p\in (1,\infty)$. 
That is, if $f_n \in W^{1,p}(\M;\X)$ satisfies that $\iota\circ f_n$ is bounded in $W^{1,p}(\M;\R^D)$, then there exists a subsequence and an $f \in W^{1,p}(\M;\X)$ such that $f_n\wc f$ in $W^{1,p}(\M;\X)$, which by definition means that $\iota\circ f_n \wc \iota\circ f$ in $W^{1,p}(\M;\R^D)$.
Even though $df$ has to be interpreted as a weak derivative, it still holds that $df_q$ is a linear map from $T_q\M$ to $T_{f(q)}\X$ for almost every $q\in\M$. 
Furthermore, the chain rule holds a.e., $D(\iota\circ f) = D\iota\circ Df$ \cite[Prop.~1.8]{CVS16}.

Our choice of working with the tangent map $Df$ rather than the differential $df$ is due to the following:
When $f$ is a Sobolev map, the pullback bundle $f^*T\X$ is a not a smooth vector bundle, and consequently $df$ is a bundle map into a non-smooth vector bundle; this is in contrast with the tangent map $Df$, which is a map between two smooth vector bundles, $T\M$ and $T\X$.
The price to pay is that $Df$ is not a linear vector bundle map, but rather a linear bundle morphism.

Note that, if $\iota:(\X,\frakx)\to(\R^D,\euc)$ is an isometric embedding, then by our convention,
\[
D\iota : T\X\to \R^D\times \R^D,
\qquad
D\iota:(q,\xi) \mapsto (\iota(q),d\iota_q(\xi)),
\]
which is an embedding of $T\X$ into Euclidean space (though not necessarily isometric),
whereas
\[
d\iota: T\X\to \R^D,
\qquad
d\iota:(q,\xi) \mapsto d\iota_q(\xi)
\]
is not an embedding.
For $f\in W^{1,p}(\M,\X)$,
\[
D(\iota\circ f) = D\iota \circ Df : T\M\to  \R^D\times \R^D,
\]
is an immersion if $f$ is an immersion, whereas
\[
d(\iota\circ f) = d\iota \circ Df : T\M\to  \R^D
\]
is not.
In this context, $T^*\M\otimes T\X$ is a vector bundle over $\M\times \X$, and the latter can be viewed as a fiber bundle over $\M$. 
A section of $T^*\M\otimes T\X\to\M$ is a map $\zeta:\M\to T^*\M\otimes T\X$, such that for $q\in\M$, $\zeta(q)$ is a linear map from $T_q\M$ to $T_{\pi(\zeta(q))}\X$. 
Moreover, we interpret $D\iota\circ \zeta: \M\to \R^D\times (T^*\M\otimes\R^D)$, i.e., $D\iota$ only acts here on the $T\X$ component.
We define
\[
L^p(\M;T^*\M\otimes T\X) = \{\zeta :\M\to T^*\M\otimes T\X ~:~ D\iota\circ \zeta \in L^p(\M;\R^D\times T^*\M\otimes\R^D)\}.
\]
Finally, for $\xi:\M\to T\X$, we denote
\[
|\xi|_\frakx = |d\iota\circ\xi|_\euc,
\]
i.e., for $q\in \M$, the norm only captures the linear part of the mapping $\xi(q)\in T_{\pi(\xi(q))}\X$.

When $p>\dim\M$, one can equivalently define $W^{1,p}(\M;\X)$ as the set of continuous functions, which upon composition with charts in $\M$ and $\X$ are in $W^{1,p}(\R^{\dim \M};\R^{\dim \X})$; strong convergence in $W^{1,p}(\M;\X)$ is equivalent to strong convergence in every coordinate chart \cite[Lemmata B.5, B.7]{Weh04}.

\paragraph{Sobolev vector fields} 

Given two compact Riemannian manifolds $(\M,\g)$ and $(\N,\h)$, we apply the above construction to Sobolev maps from $\M$ into vector fields in $\N$, i.e., to maps $\xi: \M\to T\N$.
To this end, we need to introduce a Riemannian metric on $T\N$.

The Riemannian metric $\h$ on $T\N$ induces in a canonical way a Riemannian metric on $T\N$.  
Specifically, there exists a one-to-one correspondence between an affine connection $\nabla$ on $T\N$ and a \Emph{connector operator}, $K: TT\N\to T\N$, such that for $q\in\N$ and $s\in T_q\N$,
\[
K_s : T_s T\N\to T_q\N.
\]
The connector induces the covariant derivative via
\[
\nabla_s \eta = K^\h \circ D\eta(s) 
\qquad
\text{for every $\eta\in\Gamma(T\N)$ and $s\in T\N$}.
\]
Denote by $x\in \R^{\dim \N}$ coordinates on $\N$; by choosing a local frame field in this coordinate patch, we obtain local coordinates $(x,v)\in  \R^{2\dim \N}$ of $T\N$.
We denote the associated coordinates on $TT\N$ by $(x,v,w,s)\in \R^{4\dim \N}$.
In these coordinates, the projection $\pi:T\N\to \N$ is given by $\pi(x,v) = x$, its derivative $D\pi:TT\N\to T\N$ by $D\pi(x,v,w,s) = (x,w)$, and the connector $K$ is given by 
\[
K(x,v,w,s)=(x, s+\Gamma(x)[v,w]),
\]
where $\Gamma:\R^{\dim \N}\to \operatorname{Bil}(\R^{\dim\N})$ is a smooth map into the space of bilinear maps on $\R^{\dim\N}$ (representing the Christoffel symbols).
In the following, we will always take the Levi-Civita connection $\nabla^\h$, and its corresponding connector $K^\h$.

The map $D\pi\times K^\h :TT\N\to T\N\times_\N T\N$ is a linear bundle isomorphism (covering the projection $\pi:T\N\to\N$), turned into an isometry by introducing the \Emph{Sasaki metric} $S^\h$ on $TT\N$ \cite{Sas58}, 
\beq\label{eq:sasaki}
\ip{V,W}_{S^\h} = \ip{D\pi(V),D\pi(W)}_\h + \ip{K^\h(V),K^\h(W)}_\h.
\eeq

By choosing an isometric embedding $\iota:(T\N,S^\h) \to (\R^D,\euc)$, where $\euc$ is the Euclidean metric,
that is,
\[
\ip{d\iota(V),d\iota(W)}_\euc = \ip{V,W}_{S^\h}
\qquad
V,W\in TT\N,
\]
we can define the Sobolev space $W^{1,p}(\M;T\N)$ as above. 

The zero section $\zeta\in\Gamma(T\N)$ is an isometric embedding of $(\N,\h)$ into  $(T\N,S^\h)$: indeed, for $q\in\N$ and $v,w\in T_q\N$, 
\[
\begin{split}
\ip{v,w}_{\zeta^\#S^\h} &= \ip{D\zeta(v),D\zeta(w)}_{S^\h} \\
&=  \ip{D\pi\circ D\zeta(v),D\pi \circ D\zeta(v)}_\h + \ip{K^\h\circ D\zeta(v),K^\h\circ D\zeta(v)}_\h \\
&= \ip{v,w}_\h,
\end{split}
\]
where the first equality is the definition of the pullback metric $\zeta^\# S^\h$, and in the last passage we used the fact that $D\pi\circ D\zeta = \id_{T\N}$ and $K^\h\circ D\zeta =0$. 
Thus, $\jmath: \N\to\R^D$ given by $\jmath = \iota\circ \zeta$ is an isometric embedding of $(\N,\h)$ into $(\R^D,\euc)$ (although for our uses one can choose any other isometric embedding of $\N$ into Euclidean space).

For a map $\xi:\M\to T\N$ we denote $f_\xi = \pi \circ \xi : \M\to \N$. 
Note that by the definition of the Sasaki metric,
\beq
|D\xi|^2 = |Df_\xi|^2 + |K^\h\circ D\xi|^2,
\label{eq:Dxi}
\eeq
or equivalently,
\beq
|d\iota\circ D\xi|^2 = |d\jmath\circ D f_\xi|^2 + |d\jmath\circ K^\h\circ D\xi|^2.
\label{eq:Dxi_coord}
\eeq
(Note that $d\jmath:T\N\to\R^D$ rather than $\iota:T\N\to\R^D$ is the linear isometry.)
The following lemma asserts that strong and weak convergence of  a sequence $\xi_n$ implies the corresponding convergence of $f_{\xi_n}$:

\begin{lemma}\label{lem:f_convergence}
Let $p\in[1,\infty)$.
If $\xi \in W^{1,p}(\M;T\N)$, then $f_\xi \in W^{1,p}(\M;\N)$.
Furthermore, if $\xi_n \wc \xi$ in $W^{1,p}(\M;T\N)$, then $f_{\xi_n} \wc f_\xi$ in $W^{1,p}(\M;\N)$, and similarly for strong convergence.
\end{lemma}

\begin{proof}
Let $\iota:(T\N,S^\h)\to(\R^D,\euc)$ and $\jmath : (\N,\h)\to (\R^D,\euc)$ be defined as above. 
Let $\xi\in W^{1,p}(\M;T\N)$. Since $\N$ is compact, the fact that $f_\xi\in W^{1,p}(\M;\N)$ follows from \eqref{eq:Dxi} (or equivalently, \eqref{eq:Dxi_coord}).

Let $\xi_n \wc \xi$ in $W^{1,p}(\M;T\N)$, i.e., $\iota\circ\xi_n \wc \iota\circ \xi$ in $W^{1,p}(\M;\R^D)$.
By \eqref{eq:Dxi_coord} and the compactness of $\N$,  $\jmath \circ f_{\xi_n}$ is bounded in $W^{1,p}(\M;\R^D)$ (and equi-integrable if $p=1$), and thus $\jmath\circ f_{\xi_n}$ weakly converges to some $F\in W^{1,p}(\M;\R^D)$ (modulo a subsequence). 
We need to show that $F = \jmath\circ f_\xi$.
Since weak $W^{1,p}$-convergence implies strong $L^p$-convergence, we can move to a subsequence and obtain that $\iota\circ\xi_n \to \iota\circ\xi$ almost everywhere, and thus, for that same subsequence, $\jmath\circ f_{\xi_n}\to \jmath\circ f_\xi$ almost everywhere, which by the uniqueness of the limit implies that $F=\jmath\circ f_\xi$.
As we could have started with any subsequence of $\xi_n$ and obtain the result for a sub-subsequence, the claim holds for the whole sequence.

Next assume that  $\xi_n \to \xi$ strongly in $W^{1,p}(\M;T\N)$.
Since, a fortiori, $\xi_n \wc  \xi$ in $W^{1,p}(\M;T\N)$, it follows from the previous clause that 
$f_{\xi_n} \wc f_\xi$ in $W^{1,p}(\M;\N)$, or equivalently, that $\jmath\circ f_{\xi_n}\wc \jmath\circ f_\xi$ in $W^{1,p}(\M;\R^D)$.
We need to show that $d(\jmath\circ f_{\xi_n})\to d(\jmath\circ f_\xi)$ in $L^p(\M;T^*\M\otimes\R^D)$.
It is sufficient to prove this for a subsequence, and thus we can assume, without loss of generality, that $D\iota \circ D\xi_n \to D\iota \circ D\xi$ almost everywhere. 
Since $D\iota$ is an embedding, it follows that $D\xi_n \to D\xi$ almost everywhere (as maps from $\M\to T^*\M\otimes TT\N$).
In particular, $D\pi\circ D\xi_n\to D\pi \circ D\xi$ almost everywhere (as maps $\M \to T^*\M\otimes T\N$), 
i.e., $D f_{\xi_n} \to D f_{\xi_n}$ almost everywhere, and thus also $d(\jmath\circ f_{\xi_n})\to d(\jmath\circ f_\xi)$.
To complete the proof it remains to show that 
\beq\label{eq:Df_n_to_Df}
\limn \int_\M |Df_{\xi_n}|^p\, \Volume = \int_\M |Df_\xi|^p\, \Volume.
\eeq
Since  $D f_{\xi_n} \to D f_{\xi_n}$ almost everywhere we also have that $|Df_{\xi_n}|^p \to |Df_\xi|^p$ almost everywhere,
and since 
\[
|Df_{\xi_n}|^p \le |D\xi_n|^p
\Textand
|Df_\xi|^p \le |D\xi|^p,
\]
and since it follows from the strong $W^{1,p}$-convergence of $\xi_n$ that
\[
\limn \int_\M |D\xi_n|^p\, \Volume = \int_\M |D\xi|^p\, \Volume,
\]
the limit \eqref{eq:Df_n_to_Df} follows from
a refinement of the dominated convergence theorem (e.g., \cite[Theorem 1.20]{EG15}).
\end{proof}

The following lemma, which will be needed in the sequel, addresses a situation where $\xi_n\wc\xi$ in $W^{1,p}$ along with $K^\h \circ D\xi_n$ converging strongly in $L^p$. 
The fact that the limit is, as one would expect, $K^\h \circ D\xi$ is somewhat non-trivial, since the two converging sequences have to be interpreted with respect to two different maps into $\R^D$. 
Since this suffices in the application below, we assume that $p$ is large enough:

\begin{lemma}\label{lem:connector_convergence}
Let $p>\dim \M$, and let $\iota$ and $\jmath$ be defined as above. Assume that $\xi_n \wc \xi$ in $W^{1,p}(\M;T\N)$.
Assume further that $d\jmath \circ K^\h\circ D\xi_n \to d\jmath \circ T$ in $L^p(\M;T^*\M\otimes\R^D)$ for some map $T:\M \to \Hom(T\M,T\N)$.
Then $T = K^\h \circ D\xi$ a.e.
\end{lemma}

\begin{proof}
Since $p>\dim\M$, $\xi_n \to \xi$ uniformly, hence we can work with local coordinates:
Let $X:\Omega \subset \R^{\dim \M} \to \M$ and $Y: \tilde\Omega \subset \R^{\dim \N} \to \N$ be local coordinate systems, and 
denote $(x_n , v_n) = (DY)^{-1} \circ \xi_n \circ X$, and similarly $(x,v) = (DY)^{-1} \circ \xi \circ X$, both maps $\Omega\to \tilde\Omega\times\R^{\dim\N}$.
The uniform convergence $\xi_n\to\xi$ implies that $(x_n,v_n) \to (x,v)$ uniformly.
The boundedness of $\xi_n$ in $W^{1,p}(\M;T\N)$ implies the boundedness of the coordinate expressions $(x_n,v_n)$ in $W^{1,p}(\Omega, \R^{2\dim \N})$, hence $(x_n,v_n) \wc (x,v)$ in $W^{1,p}(\Omega, \R^{2\dim \N})$ to $(x,v)$.
In these coordinates, 
\[
D\xi_n = (x_n,v_n, Dx_n[\cdot], Dv_n[\cdot]): \Omega \to (\tilde\Omega\times\R^{\dim\M}) \times \Hom(\R^{\dim\N}, \R^{2\dim \N}), 
\]
hence,
\[
K^\h \circ D\xi_n = (x_n,Dv_n[\cdot] + \Gamma(x_n)[v_n,Dx_n[\cdot]]).
\]
Since $Dv_n \wc Dv$ and $Dx_n \wc Dx$ in $L^p$, $\Gamma$ is smooth and $(x_n,v_n)\to (x,v)$ uniformly, we obtain that
\[
(x_n,Dv_n[\cdot] + \Gamma(x_n)[v_n,Dx_n[\cdot]]) \wc  (x,Dv[\cdot] + \Gamma(x)[v,Dx[\cdot]])
\]
in $L^p(\Omega;\tilde\Omega\times\Hom(\R^{\dim\M},\R^{\dim\N}))$.
We has thus proved that $K^\h \circ D\xi_n \wc K^\h \circ D\xi$ weakly in $L^p$ in every coordinate patch.

Add now the fact that $dj \circ K^\h\circ D\xi_n \to dj \circ T$ strongly in $L^p$.
By moving to subsequences, we can assume that $dj \circ K^\h\circ D\xi_n \to dj \circ T$ a.e.
Since $Dj:\N\to \R^{2K}$ is an embedding, it follows that $K\circ D\xi_n \to T$ a.e.
Therefore, in coordinates, for almost every $p\in \Omega$,
\[
(x_n(p),D_pv_n[\cdot] + \Gamma(x_n(p))[v_n(p),D_px_n[\cdot]]) \to T(p)[\cdot].
\]
Thus $T = K \circ D\xi$ a.e. in every coordinate patch, hence a.e.
\end{proof}

The following is an immediate application of \cite[Prop.~D.1]{AKM22}:

\begin{proposition}
\label{prop:D1}
Let $p>\dim\M$. Then, $\xi_n\to \xi$ in $W^{1,p}(\M;T\N)$ if and only if $f_{\xi_n}\to f_\xi$ in $W^{1,p}(\M;\N)$, and for some isometric embedding $\jmath:\N\to\R^D$, 
\[
d\jmath\circ \xi_n \to d\jmath\circ \xi 
\qquad
\text{in $L^p(\M;\R^D)$}, 
\]
and
\[
d\jmath\circ K^\h\circ D\xi_n \to d\jmath\circ K^\h \circ D\xi 
\qquad
\text{in $L^p(T^*\M;\R^D)$}.
\]
\end{proposition}

\paragraph{Sobolev immersions}
Let $\dim\M=d$ and $\dim\N=d+1$,
and let $f\in W^{1,p}(\M;\N)$ satisfy $\rank df=d$ a.e. 
We may define a.e. a map $\hn_f\in L^\infty(\M;T\N$) covering $f$, such that $\hn_f(q)\in T_{f(q)}\N$ is a unit vector, normal to the image of $Df_q$, and for every oriented basis $(u_1,\dots,u_d)\subset T_q\M$, 
\[
(df_q(u_1),\dots,df_q(u_d),\hn_f(q))\subset T_{f(q)}\N
\]
is an oriented basis.
If $\hn_f \in W^{1,p}(\M;T\N)$, we say that $f$ belongs to the set of \Emph{$p$-Sobolev immersions} (this is the definition in  \eqref{eq:imm_p}). 
Since there is a one-to-one correspondence between $f$ and $\hn_f$, $\Imm_p(\M;\N)$ can be identified with a subset of 
$W^{1,p}(\M;T\N)$, and specifically a subset of $W^{1,p}(\M;S\N)$, where $S\N\subset T\N$ is the sphere bundle of $\N$. 
It is however not a closed subset, as the rank condition is not closed under $W^{1,p}$-convergence.

For $f\in\Imm_p(\M;\N)$, the total energy \eqref{eq:Etotal} is more rigorously written as
\beq
\E_p(f) = \int_\M \brk{\dist^p(Df,\O(\g,\h)) + |Df\circ S + K^\h\circ D\hn_f|^p}\,\Volume,
\label{eq:Etotal2}
\eeq
as this notation does not rely on non-smooth pullback bundles and their corresponding pullback connections.

\paragraph{Trivializations} 

Let $\xi_n\in W^{1,p}(\M;T\N)$ be a sequence converging to $\xi\in W^{1,p}(\M;T\N)$ uniformly (hence in particular, $f_{\xi_n} \to f_\xi$ uniformly).
The compactness of $\M$ and $\N$ implies the existence of finite open covers $\calU = \{U_i\}$ and $\calV = \{V_i\}$ of $\M$ and $\N$ such that $T\N|_{V_i}$ is a trivial bundle for every $i$, and such that $f_{\xi_n}(U_i)\subset V_i$ for $n$ large enough. 
For every such $U$ and $V$ (to simplify notations we omit the index $i$), let $R:V\to\SO(\h,\R^{\dim\N})$ be a smooth orthonormal frame (which may or may not be a coordinate frame), namely, for every $q\in V$ and $\eta,\xi\in T_q\N$,
\[
\ip{\eta,\xi}_\h = \ip{R \circ \eta,R \circ \xi}_{\euc}.
\]
Such a frame exists by the triviality of $T\N|_V$.
Then, restricting to $U$, 
\[
R \circ Df_{\xi_n}, R \circ Df_\xi \in L^p\Omega^1(U;\N\times \R^{\dim\N}),
\]
where $R$ acts on the vector part $df_{\xi_n}$ of $Df_{\xi_n}$.
Note that $R\circ Df_{\xi_n} - R\circ Df_\xi \ne R\circ(Df_{\xi_n} - Df_\xi)$; in fact, the right-hand side is not well-defined, as the images of $\xi_n$ and $\xi$ do not belong to the same fiber of $T\N$. 

\begin{proposition}
\label{prop:trivialization}
Assume that $p>\dim \M$. Then,
\[
f_{\xi_n} \to f_\xi 
\qquad\text{in $W^{1,p}(\M;\N)$}
\]
if and only if
\[
R\circ Df_{\xi_n} \to R \circ Df_\xi 
\qquad\text{in $L^p\Omega^1(U;\N\times\R^{\dim\N})$}
\]
for every $U\in\calU$.
Likewise, if 
\[
f_{\xi_n} \wc f_\xi 
\qquad\text{in $W^{1,p}(\M;\N)$}
\]
then
\[
R\circ Df_{\xi_n} \wc R \circ Df_\xi 
\qquad\text{in $L^p\Omega^1(U;\N\times\R^{\dim\N})$}
\]
for every $U\in\calU$.
\end{proposition}

\begin{proof}
Since $R\circ Df_{\xi_n}(x) = (f_{\xi_n}(x), R_{f_{\xi_n}(x)} \circ (df_{\xi_n})_x)$, and $R_q$ is a linear map depending smoothly on the footpoint, and the footpoints $f_{\xi_n}$ converge uniformly to $f_\xi$ by assumption, the proof follows from the same arguments as the one of Proposition~\ref{lem:connector_convergence}.
\end{proof}

\subsection{Coordinate representation of the energy}\label{sec:coordinates}

In this section, we present some of geometric constructs and the energy \eqref{eq:Etotal} in coordinates for the benefit of readers who are more used to this notation. 
For simplicity, assume that both $\M$ and $\N$ are covered by single coordinate charts, 
\[
x_\M:\M\to\R^d
\Textand
x_\N:\N\to\R^{d+1}.
\]
A map $f:\M\to\N$ is represented by a map $\R^d\supset x_\M(\M) \to \R^{d+1}$, 
\[
x \mapsto x_\N\circ f \circ x_\M^{-1},
\]
with components $f^\alpha$, where for the sake of clarity, we use Latin indexes for the body manifold $\M$, and Greek indexes for the space manifold $\N$.

Throughout this section, we will denote the projection of a vector bundle onto its base by $\pi$, where the type (e.g., $T\M\to\M$ or $TT\N\to T\N$) should be clear from the context.
The tangent bundle $T\M$ has coordinates $(x_{\M},\dot{x}_{\M}):T\M\to\R^d\times\R^d$,  where, with a slight abuse of notations which will be used repeatedly, $x_{\M} = x_\M\circ\pi$  and $\dot{x}_{\M} = dx_\M$; the metric $\g$ is represented by a matrix-valued function $x_\M(\M)\to \R^d\otimes\R^d$, with entries $\g_{ij}$, such that for $V \in T\M$,
\[
|V|^2 = \g_{ij}(x_{\M}(V))\, \dot{x}_{\M}^i(V) \dot{x}_{\M}^j(V),
\] 
with the Einstein summation convention over repeated indexes. 
Similarly, the tangent bundle $T\N$ has coordinates $(x_{\N},\dot{x}_{\N}):T\N\to\R^{d+1}\times\R^{d+1}$; the metric $\h$ is  represented by a matrix-valued function $x_\N(\N)\to\R^{d+1}\otimes\R^{d+1}$, with entries $\h_{\alpha\beta}$, so that for a vector $W\in T\N$,
\[
|W|^2 = \h_{\alpha\beta}(x_{\N}(W))\, \dot{x}_{\N}^\alpha(W) \dot{x}_{\N}^\beta(W).
\] 

For $f:\M\to\N$, the fiber of the pullback bundle $f^*T\N$ at a point $p\in\M$ is canonically identified with the fiber $T_{f(p)}\N$; its coordinates are $(x_{\M},\dot{x}_{\N}):f^*T\N\to\R^d\times\R^{d+1}$, where $x_{\M} = x_{\M}\circ\pi$ and $\dot{x}_{\N}$ 
is identified with $\dot{x}_\N: T\N\to\R^{d+1}$, via the canonical identification of fibers of $f^*T\N$ with fibers of $T\N$. The coordinate representation of $df\in\Gamma(f^*T\N)$ is the map $x_\M(\M)\to \Hom(\R^d,\R^{d+1})$ given by
$\partial_if^\alpha$, 
whereas the coordinate representation of $Df:T\M\to T\N$ is the map $x_\M(\M)\to \R^{d+1}\times\Hom(\R^d,\R^{d+1})$, given by $(f^\alpha,\partial_if^\alpha)$.

We denote by $\g^{1/2}: x_\M(\M)\to\Hom(\R^d,\R^d)$ and $\h^{1/2}:x_\N(\N)\to\Hom(\R^{d+1},\R^{d+1})$ the symmetric positive-definite square roots of the matrices $\g$ and $\h$, i.e.,
\[
\g_{ij} = \delta_{kl} (\g^{1/2})_i^k  (\g^{1/2})_j^\ell
\Textand
\h_{\alpha\beta} = \delta_{\gamma\eta} (\h^{1/2})_\alpha^\gamma  (\h^{1/2})_\beta^\eta.
\]
The stretching energy of $f:\M\to\N$ takes the form
\[
\E_p^S(f) = \int_{x_\M(\M)} \dist^p(Q(x),O(d)) \det(\g^{1/2}(x))\, dx,
\]
where
\[
Q^\alpha_i(x) = (\h^{1/2}(f(x)))_\beta^\alpha \, \partial_j f^\beta(x) \,(\g^{-1/2}(x))_i^j,
\]
or, in matrix form,
\[
Q(x) = h^{1/2}(f(x)) \circ df(x) \circ \g^{-1/2}(x),
\]
and the distance here is with respect to the Euclidean Frobenius norm. 

We proceed with the double tangent $TT\N$, which has coordinates
\[
(x_{\N},\dot{x}_{\N},v_{\N},\dot{v}_{\N}):TT\N\to\R^{d+1}\times\R^{d+1}\times\R^{d+1}\times\R^{d+1},
\]
where $x_{\N} = x_{\N}\circ\pi$, $\dot{x}_{\N} = \dot{x}_{\N}\circ \pi$, $v_{\N} = \dot{x}_{\N}\circ D\pi$ and $\dot{v}_{\N} = d\dot{x}_{\N}$.
That is, a vector in $T_{(x_\N,\dot{x}_\N)}T\N$ represents a direction of change, in $T\N$, from the point $(x_\N,\dot{x}_\N)$; the coordinate $v_{\N}$ represents the direction of change in the coordinate $x_\N$, and $\dot{v}_{\N}$ the direction of change in the coordinate $\dot{x}_{\N}$.

We further denote by $\Gamma: x_\N(\N)\to \Hom(\R^{d+1}\otimes\R^{d+1},\R^{d+1})$  the Christoffel symbols of the Levi-Civita connection, with coordinates $\Gamma^\alpha_{\beta\gamma}$; the connection map $K^\h:TT\N\to T\N$ is defined by,
\[
x_{\N}\circ K^\h = x_{\N}
\Textand
\dot{x}_{\N}\circ K^\h = \dot{v}_{\N} + \Gamma(x_{\N})[\dot{x}_{\N},v_{\N}].
\]

Let $\xi,Y\in \frakX(\N)$ be vector fields having coordinate representation $x_\N(\N)\to \R^{d+1}$,
\[
\xi^\alpha = \dot{x}_\N^\alpha\circ \xi \circ x_{\N}^{-1}
\Textand
Y^\alpha = \dot{x}_\N^\alpha\circ Y \circ x_{\N}^{-1}.
\]
The coordinate representation of $D\xi(Y):\N\to TT\N$ is the map $x_\N(\N) \to \R^{d+1}\times \R^{d+1}\times \R^{d+1}$, given by
\[
(Y^\alpha,\xi^\alpha, Y^\beta\, \partial_\beta\xi^\alpha).
\]
Thus, the coordinate representation of $K^\h\circ D\xi(Y):\N\to T\N$ is the map $x_\N(\N)\to\R^{d+1}$, given by
\[
Y^\beta\, \partial_\beta\xi^\alpha +  \Gamma^\alpha_{\beta\gamma}(x) Y^\beta \xi^\gamma,
\]
thus showing that indeed $K^\h\circ D\xi(Y) = \nabla_Y\xi$.

We proceed to write the bending energy in explicit form. The reference shape operator $S$ is represented by a map $x_\M(\M)\to\Hom(\R^d,\R^d)$ with indexes $S_i^j$. The unit normal $\hn_f$ is represented by a function $x_\M(\M)\to \R^{d+1}$, such that
\[
(\partial_1 f^\alpha(x),\dots,\partial_df^\alpha(x),\hn_f^\alpha(x))
\]
is an oriented basis for $\R^{d+1}$, and $\hn_f(x)$ is a unit vector normal to range of $\nabla f(x)$ with respect to the inner-product $\h_{\alpha\beta}(f(x))$.  
With this we have
\[
\E_p^B(f) = \int_{x_\M(\M)} \brk{\g^{ij}(x) \h_{\alpha\beta}(f(x)) A_i^\alpha(x) A_j^\beta(x)}^{p/2} \,\det(\g^{1/2}(x))\, dx,
\]
where 
\[
A_i^\alpha(x) = \partial_i \hn_f^\alpha(x) + \Gamma^\alpha_{\beta\gamma}(f(x)) \partial_i f^\beta(x) \hn_f^\gamma(x) + \partial_i f^\gamma(x) S_\gamma^\alpha(x),
\]
represents the discrepancy between the target shape operator $S$ and the shape operator of $f$ as defined in \eqref{eq:S_f}.

\section{Proof of \thmref{thm:main}}\label{sec:main_pf}

\subsection{The relaxed functional}\label{sec:relaxed}

Denote as above $d = \dim \M$.
Consider the vector bundle of rank $d+1$, $T\M\oplus \R\to\M$ endowed with the product metric 
\[
\ip{(v,s),(w,t)}_G = \ip{v,w}_\g + st,
\qquad\text{for $q\in\M$, $v,w\in T_q\M$ and $s,t\in\R$}.
\]
The orientation of $T\M\oplus\R$ is induced by the orientation of $T\M$: if $(u_1,\dots,u_d)$ is an oriented basis for $T_q\M$, then, $(u_1,\dots,u_d,1)$ is an oriented basis for $T_q\M\oplus\R$.

For $\xi\in W^{1,p}(\M;T\N)$ we introduce the map
\[
Df_\xi\oplus \xi \in L^p(\M;(T\M \oplus \R)^*\otimes T\N),
\] 
which is a linear bundle morphism covering $f_\xi$ defined a.e.\ by
\[
(Df_\xi\oplus \xi)_q(v,t) = d_qf_\xi(v) + t\, \xi(q) \in T_{f_\xi(q)}\N,
\]  
where $q\in\M$, $v\in T_q\M$ and $t\in \R$. That is, on the left-hand side we identify $\xi(q)\in T_{f_\xi(q)}\N$ with $\xi(q)\in\Hom(\R,T_{f_\xi(q)}\N)$.

\begin{definition}\label{def:relax_en}
The \Emph{relaxed energy functional} $\tE_p : W^{1,p}(\M;T\N)\to [0,\infty)$ is defined by
\[
\begin{split}
\tE_p(\xi)  &= \int_\M \dist^p(Df_\xi\oplus \xi,\SO(G,\h))\,\Volume + 
\int_\M |Df_\xi\circ S + K^\h\circ D\xi|^p\,\Volume \\
&\equiv \tE_p^S(\xi) + \tE_p^B(\xi).
\end{split}
\]
\end{definition}

\begin{proposition}
\label{prop:tE_p}
The functional $\tE_p$ is a relaxation of $\E_p$, in the sense that
\[
\E_p(f) = \tE_p(\hn_f) 
\qquad
\text{for all $f\in \Imm_p(\M,\N)$}.
\]
Furthermore $\tE_p^S(\xi) = 0$ if and only if $f_\xi \in \Imm_p(\M;\N)$, $\xi= \hn_{f_\xi}$ and $Df\in \O(\g,\h)$ a.e.
\end{proposition}

\begin{proof}
We first prove the second claim. Let $\tE_p^S(\xi) = 0$, i.e., $Df_\xi\oplus\xi\in \SO(G,\h)$ a.e.
Let $v\in T\M$. Then,
\[
|Df_\xi(v)|_\h = |(Df_\xi\oplus\xi)(v,0)|_\h = |(v,0)|_G = |v|_\g, 
\]
thus proving that $Df_\xi \in \O(\g,\h)$. In particular, $Df_\xi$ has full rank a.e. Moreover, a.e., and for every $v\in T\M$,
\[
\ip{Df_\xi(v),\xi}_\h = \ip{(Df_\xi\oplus\xi)(v,0),(Df_\xi\oplus\xi)(0,1)}_\h = \ip{(v,0),(0,1)}_G = 0,
\]
whereas
\[
\ip{\xi,\xi}_\h = \ip{(Df_\xi\oplus\xi)(0,1),(Df_\xi\oplus\xi)(0,1)}_\h = \ip{(0,1),(0,1)}_G = 1,
\]
i.e., $\xi = \hn_{f_\xi}$ and since $\xi\in W^{1,p}(\M;T\N)$, it follows that $f_\xi\in \Imm_p(\M;\N)$.  The other direction is immediate.

The first claim follows from the next lemma.
\end{proof}

\begin{lemma}
For $f\in \Imm_p(\M,\N)$, the following identity holds,
\[
\dist(Df\oplus \hn_f,\SO(G,\h)) = \dist(Df,\O(\g,\h)).
\]
\end{lemma}

\begin{proof}
This is a linear-algebraic statement. 
Let $(V_1,\g_1)$, $(V_2,\g_2)$ and $(W,\h)$ be oriented $d$-, $k$- and $(d+k)$-dimensional inner-product spaces. Let $A\in\Hom(V_1,W)$ and $B\in\Hom(V_2,W)$ satisfy
\begin{enumerate}[itemsep=0pt,label=(\alph*)]
\item $A\oplus B \in \Hom(V_1\oplus V_2,W)$ is orientation-preserving (with respect with the natural orientation of $V_1\oplus V_2$).
\item $\image(A) =  (\image(B))^\perp$.
\item $B\in \O(\g_2,\h)$.
\end{enumerate}
Then, 
\[
\dist(A\oplus B,\SO(\g_1\oplus \g_2,\h)) = \dist(A,\O(\g_1,\h)).
\]
In the present case $k=1$, $A = Df$ and $B = \hn_f$.
Indeed, by definition, $Df\oplus \hn_f$ is orientation-preserving, $\hn_f$ is a unit vector, and $G$ is a product metric.
\end{proof}

The fact that $\E_p(f) = \tE_p(\hn_f)$  implies that \thmref{thm:main} is proved if we prove the following relaxed version:

\begin{theorem}
\label{thm:main2}
Suppose that there exists a sequence $\xi_n \in W^{1,p}(\M;T\N)$ satisfying 
\[
\limn \tE_p(\xi_n) \to 0.
\] 
Then there exists a subsequence of $\xi_n$ converging strongly in $W^{1,p}(\M;T\N)$ to a smooth limit $\xi$. Furthermore, 
$\tE_p(\xi)=0$, implying that $f_\xi\in\Imm_p(\M;\N)$ is a smooth isometric immersion, satisfying
$K^\h\circ D\hn_{f_\xi} = - Df_\xi\circ S$, i.e., $S$ is the shape operator of the limit.
\end{theorem}
 
\subsection{Basic compactness considerations}

\begin{lemma}
Let $p>1$, and assume $\tE_p(\xi_n) \to 0$.
Then $\xi_n$ has a weakly converging subsequence $\xi_n \wc \xi$ in $W^{1,p}(\M;T\N)$.
In particular, if $p>d$ then $\xi_n\to \xi$ uniformly.
The same (weak/uniform convergence) holds for $f_{\xi_n}\to f_{\xi}$.
\end{lemma}

\begin{proof}
Let $\tE_p(\xi_n) \to 0$.
The uniform boundedness of the stretching energy $\tE_p^S(\xi_n)$ and the compactness of $\N$ imply that $\xi_n$ is a bounded sequence in $L^p(\M,T\N)$ and that $Df_{\xi_n}$ is a bounded sequence in $L^p(\M;T^*\M\otimes T\N)$:
Indeed, 
\[
|Df_\xi\oplus \xi| \ge \max\{|Df_\xi|, |\xi|\},
\]
and  
\[
\dist(Df_\xi\oplus \xi,\SO(G,\h)) \ge |Df_\xi\oplus \xi| - \sqrt{d+1},
\]
since every element of $\SO(G,\h)$ is of norm $\sqrt{d+1}$.
The boundedness of the bending energy $\tE_p^B(\xi_n)$ implies in turn that $K^\h\circ D\xi_n$ is a bounded sequence in $L^p(\M;T^*\M\otimes T\N)$.
By the definition \eqref{eq:sasaki} of the Sasaki metric, 
this implies that $D\xi_n$ is a bounded sequence in $L^p(\M;T^*\M\otimes TT\N)$, i.e., 
$\xi_n$ is bounded in $W^{1,p}(\M;T\N)$.
The claim on $\xi_n$ follows immediately, and the claim on $f_{\xi_n}$ follows from Lemma~\ref{lem:f_convergence}.
\end{proof}

\subsection{The case $p>d$}

We first prove \thmref{thm:main2} for $p>d$; this restriction is relaxed further below.
The analysis for $p>d$ can also be applied directly to the functional $\E_p:\Imm_p(\M;\N)\to [0,\infty)$; the extension to $p\le d$ in the next part, however, cannot.

\paragraph{Analysis of the limiting map}
Let $\calU$, $\calV$ and $R$ define trivializations of $T\N$  as in \secref{sec:imm_p} and let $U\in\calU$; note that $TU = T\M|_U$ is a vector bundle over $U$.
Then,
\[
\alpha_n = R\circ (Df_{\xi_n}\oplus \xi_n) \in L^p\Gamma((TU\oplus\R)^*\otimes\R^{d+1}),
\]
which is a section of a vector bundle over $U$,
is a trivialization of $Df_{\xi_n}\oplus \xi_n$ satisfying
\[
\dist(Df_{\xi_n}\oplus \xi_n,\SO(G,\h)) = \dist(\alpha_n,\SO(G,\euc)).
\]
The fact that $\tE_p(\xi_n) \to 0$ implies in particular that
\beq
\limn \int_U \dist^p(\alpha_n,\SO(G,\euc))\,\Volume = 0.
\label{eq:distG1}
\eeq

It follows from \eqref{eq:distG1} that $\alpha_n$ is an $L^p$-bounded sequence of sections of $(TU\oplus\R)^*\otimes\R^{d+1}$.
By the fundamental theorem of Young measures (see \cite[Section~3.1]{KMS19} for a version for vector bundles), there exists an $L^p$-Young measure $\nu\in Y^p(U,(TU\oplus\R)^*\otimes\R^{d+1})$ such that 
\[
\alpha_n \xrightarrow{Y} \nu.
\]
That is, for every vector bundle $E\to U$ and every (generally nonlinear) bundle map  $W:(TU\oplus\R)^*\otimes\R^{d+1}\to E$ for which $W(\alpha_n):U\to E$ is $L^1$-weakly precompact,
\[
W(\alpha_n) \wc 
\BRK{q \to \int_{(TU\oplus\R)^*\otimes\R^{d+1}} W(A) \, d\nu_q(A)} 
\qquad \text{ in $L^1\Gamma(E)$}.
\]

Take $W(A)=\dist(A,SO(G,\euc))$. Since $W(\alpha_n)$ is bounded in $L^p(U)$ for $p>1$, it is $L^1$-weakly precompact, from which follows that
\[
\begin{split}
0 &= \limn \int_U \dist(\alpha_n,\SO(G,\euc))\,\Volume \\\
&=  \int_U \int_{(TU\oplus\R)^*\otimes\R^{d+1}} \dist(A,\SO(G,\euc))\, d\nu(A)\, \Volume,
\end{split}
\]  
i.e., $\nu_q$ is supported on $\SO(G,\euc)$ for a.e. $q\in U$.

For general oriented inner-product spaces $(V,\g)$ and $(W,\h)$ of equal dimension $s$, and a linear map $A\in\Hom(V,W)$, the determinant and the cofactor of $A$, $\det A \in \Hom(\R;\R) \simeq \R$ and $\cofop A \in \Hom(V,W)$ can be defined in a basis-independent manner by
\[
\det A = \star_\h \circ (\wedge^s A)\circ \star_\g
\Textand
\cofop A = (-1)^{s+1} \star_\h \circ (\wedge^{s-1} A)\circ \star_\g,
\]
where $\star_\g:\Lambda^k V\to \Lambda^{s-k}V$ and $\star_\h:\Lambda^k W\to \Lambda^{s-k}W$ are the hodge-dual isomorphisms of the respective spaces, and $\wedge^k A$ is the $k$-minor of $A$, defined by
\[
\wedge^k A (v_1,\dots,v_k) = A(v_1) \wedge \dots\wedge A(v_k) .
\]  
These definitions can be equivalently obtained by choosing positive orthonormal bases for $V$ and $W$ and evaluating the determinant and cofactor of their matrix representations in these bases.
It is well-known that
\[
A\in \SO(\g,\h)
\qquad\text{if and only if}\qquad
\cofop A = A
\textand
\det A = 1.
\]

With that, consider next the test functions $W(A) = A$, $W(A) = \cofop A$ and $W(A)=\det A$. 
By definition, they are compositions of minors of $A$ of ranks 1, $d$ and $(d+1)$, respectively, and hodge-dual isometries.
Since $\h$ and $\h^{-1}$ are bounded uniformly (in coordinates) and $Df_{\xi_n}$ is bounded in $L^p$, both the cofactor and the determinant of $\alpha_n$ are uniformly bounded in $L^{p/d}$, and therefore $W(\alpha_n)$ are $L^1$-weakly precompact for all above choices of $W$. It follows from the definition of the Young measure limit that
\begin{align*}
\alpha_n &\wc \BRK{q \to \int_{(TU\oplus\R)^*\otimes\R^{d+1}} A \, d\nu_q(A)}  &\text{ in $L^1\Gamma((TU\oplus\R)^*\otimes\R^{d+1})$}  \\
\cofop(\alpha_n) &\wc  \BRK{q \to \int_{(TU\oplus\R)^*\otimes\R^{d+1}} \cofop A \, d\nu_q(A)}  & \text{ in $L^1\Gamma((TU\oplus\R)^*\otimes\R^{d+1})$}  \\
\det(\alpha_n) &\wc  \BRK{q \to \int_{(TU\oplus\R)^*\otimes\R^{d+1}} \det A \, d\nu_q(A)} &  \text{ in $L^1(U)$} . \\
\end{align*}

Since $\nu_q$ is supported on $\SO(G,\euc)$ for a.e. $q\in U$, and on that set $\cofop A=A$ and $\det A=1$, it follows that 
\begin{align}
\alpha_n &\wc \BRK{q \to \int_{\SO(G,\euc)} A \, d\nu_q(A)} & \text{ in $L^1\Gamma((TU\oplus\R)^*\otimes\R^{d+1})$} \nonumber \\
\cofop(\alpha_n) &\wc  \BRK{q \to \int_{\SO(G,\euc)}  A \, d\nu_q(A)} & \text{ in $L^1\Gamma((TU\oplus\R)^*\otimes\R^{d+1})$}  \label{eq:YM1}
\\
\det(\alpha_n)& \wc  1 & \text{ in $L^1(U)$} . \nonumber
\end{align}

Recall that $\xi_n \wc \xi$  (hence also $f_{\xi_n} \wc f_\xi$) in $W^{1,p}$. It follows from the second clause of \propref{prop:trivialization} that
\beq
\alpha_n \wc \alpha = R\circ(Df_\xi\oplus\xi)
\qquad\text{in $L^p\Gamma((TU\oplus\R)^*\otimes\R^{d+1})$}.
\label{eq:alpha_n_alpha}
\eeq

\begin{lemma}
The weak $L^p$-convergence $\alpha_n\wc\alpha$ implies that
\beq
\begin{aligned}
\cofop(\alpha_n) &\wc  \cofop(\alpha) \qquad \text{ in $L^{p/d}\Gamma((TU\oplus\R)^*\otimes\R^{d+1})$}  \\
\det(\alpha_n) &\wc \det(\alpha)  \qquad \text{ in $L^{p/d}(U)$} .
\end{aligned}
\label{eq:lem3.1}
\eeq
\end{lemma}

\begin{proof}
Since the hodge-dual operator is a linear isometry and $\xi_n\to\xi$ uniformly, it suffices to show that 
\[
\wedge^k \alpha_n \wc  \wedge^k \alpha \qquad \text{ in $L^{p/d}\Gamma(\Lambda^k(TU\oplus\R)^*\otimes \Lambda^k \R^{d+1})$}  
\]
for $k=d,d+1$. 
We show more generally that for every $k=1,\dots,d+1$
\[
\wedge^k \alpha_n \wc \wedge^k \alpha 
\qquad
\text{in $L^{p/\min(k,d)}\Gamma(\Lambda^k(TU\oplus\R)^*\otimes\Lambda^k \R^{d+1})$}.
\]
We start by noting that for $v_1,\ldots,v_k \in T_qU$ and $t_1,\dots,t_k\in\R$,
\[
\begin{split}
\wedge^k\alpha_n((v_1,t_1),\dots,(v_k,t_k)) &= \wedge_{j=1}^k (R_{f_{\xi_n}(q)}\circ D_qf_{\xi_n}(v_j) + t_j R_{f_{\xi_n}(q)}\circ\xi_n(q)) \\
&= \wedge_{j=1}^k R_{f_{\xi_n}(q)}\circ D_qf_{\xi_n}(v_j) \\
&\qquad + \sum_{\ell=1}^k (-1)^\ell t_\ell \wedge_{j\ne \ell} \brk{R_{f_{\xi_n}(q)}\circ D_qf_{\xi_n}(v_j)}\wedge (R_{f_{\xi_n}(q)}\circ \xi_n(q)),
\end{split}
\]
i.e.,
\[
\wedge^k\alpha_n = \wedge^k (R\circ Df_{\xi_n}\oplus 0) + 
\wedge^{k-1} (R\circ Df_{\xi_n}\oplus 0) \wedge (0\oplus R\circ \xi_n).
\]
Since $\xi_n\to\xi$ uniformly, and since the product of a weakly $L^p$-convergent sequence and a uniformly convergence sequence is weakly $L^p$-convergent, it suffices to show that
\[
\wedge^k (R\circ Df_n) \wc \wedge^k (R\circ Df)
\qquad  
\text{in $L^{p/\min(k,d)}\Omega^k(U;\Lambda^k \R^{d+1})$}.
\]
Since $\M$ and $\N$ are compact, the orthonormal frame $R$ can be replaced by any other smooth frame, for example, $R = dh$, where $h:V\to\R^{d+1}$ is a local coordinate system. The claim reduces then to the standard weak-continuity of minors, proved inductively on  $k$ (see, e.g., \cite[Lemma~5.10]{Rin18}).
\end{proof}

Since $\cofop(AB) = \cofop(A)\,\cofop(B)$ and $\det(AB) = \det(A)\,\det(B)$, and since $R$ is an orientation-preserving isometry, 
\begin{align*}
\cofop(R\circ A) &= R\circ \cofop(A)\\
\det(R\circ A) &= \det(A),
\end{align*}
we obtain by combining \eqref{eq:lem3.1} with \eqref{eq:YM1} that,
\[
\cofop(Df_\xi\oplus \xi) = Df_\xi\oplus \xi
\Textand
\det(Df_\xi\oplus \xi) = 1 \qquad\text{a.e.,}
\]
which implies that
\[
Df_\xi\oplus \xi \in \SO(G,\h) \qquad\text{a.e.}
\]
Thus $\tE_p^S(\xi) = 0$. By \propref{prop:tE_p}, $f_\xi \in \Imm_p(\M;\N)$ with $Df_\xi\in \O(\g,\h)$, and $\xi = \hn_{f_\xi}$. 

Thus, we have obtained  at this stage  that $\xi_n \wc \xi$ in $W^{1,p}$ where, $\xi = \hn_{f_\xi}$.
In order to complete the proof for the case $p>d$, we need to show that the convergence is in fact strong, that the shape operator of $\hn_{f_\xi}$ is $S$, and deduce from Theorem~\ref{thm:reg} that the limit is smooth.

\paragraph{Strong convergence and second fundamental form of the limit}
From the uniqueness of the limit, combining \eqref{eq:YM1} and \eqref{eq:alpha_n_alpha},
\[
\int_{\SO(G_x,\euc)} A \, d\nu_x(A) = \alpha_x
\qquad
\text{for a.e. $x\in U$}.
\] 
The left-hand side is a convex combination of elements in $\SO(G_x,\euc)$ and the right-hand side is in $\SO(G_x,\euc)$  ($\alpha$ is the composition of an element of $\SO(G,\h)$ and an element of $\SO(\h,\euc)$).
Since $\SO(G_x,\euc)$ is strictly convex, this convex combination must be trivial, namely,  
\[
\nu_x = \delta_{\alpha_x}
\qquad
\text{for a.e. $x\in U$}.
\]
As a consequence, for every $W:(TU\otimes\R)^*\otimes\R^{d+1}\to\R$ for which 
$W(\alpha_n)$ is $L^1$-weakly precompact,
\beq
\limn \int_U W(\alpha_n)\,\Volume = \int_U W(\alpha)\,\Volume.
\label{eq:YM3}
\eeq

We now show that $f_{\xi_n}\to f_\xi$ strongly in $W^{1,p}(\M;\N)$ (for the moment, we have only established weak convergence, hence in particular strong $L^p$ convergence).
By the first clause of \propref{prop:trivialization}, it suffices to prove that $\alpha_n\to\alpha$ in $L^p\Gamma((TU\otimes\R)^*\otimes\R^{d+1})$. We would be done if we could use \eqref{eq:YM3} for
\[
W(A) = |A - \alpha|^p,
\]
however, $W(\alpha_n)$ is only bounded in $L^1(U)$, which does not guarantee sequential weak precompactness. 
Instead, let 
\[ 
W(A)=|A- \alpha|^p \,\vp\brk{\frac{|A|}{3\sqrt{d+1}}},
\]
where $\vp:[0,\infty)\to\R$ is continuous, non-negative, compactly-supported, and satisfies $\vp(t)=1$ for $t\le1$ and $\vp(t)<1$ for $t>1$.
Clearly, $W(\alpha_n)$ is uniformly bounded, hence $L^1$-weakly precompact, which since $W(\alpha)=0$ implies that
\[
\limn \int_U W(\alpha_n) \,\Volume =  0.
\]
On the set
\[
B_n = \{x\in U ~:~ |\alpha_n|<3\sqrt{d+1}\},
\]
we have $W(\alpha_n) = |\alpha_n - \alpha|^p$, whereas on its complement $U\setminus B_n$,
\[
|\alpha_n - \alpha|  \le  2(|\alpha_n| - |\alpha|) \le 2\,\dist(\alpha_n,\SO(G,\euc)),
\]
where the first inequalities follow from the fact that $|\alpha| = \sqrt{d+1}$ and $|\alpha_n|\ge 3\sqrt{d+1}$, and the second inequality follows from the reverse triangle inequality.
Thus, in $U\setminus B_n$,
\[
|\alpha_n - \alpha|^p \le 2^p\,\dist^p(\alpha_n,\SO(G,\euc)).
\]
Combining the two sets,
\[ 
\begin{split}
\int_U |\alpha_n- \alpha|^p\,\Volume &= \int_{B_n} |\alpha_n- \alpha|^p\,\Volume + \int_{U\setminus B_n} |\alpha_n- \alpha|^p\,\Volume \\
&\le   \int_U W(\alpha_n)\,\Volume + 2^p \int_{U} \dist^p(\alpha_n,\SO(G,\euc))\,\Volume,
\end{split}
\]
and both term tend to zero as $n\to\infty$. We conclude that $Df_{\xi_n} \to Df_{\xi}$ in $L^p$ and therefore obtain that $f_{\xi_n} \to f_{\xi}$ in $W^{1,p}(\M;\N)$.

It remains to show that $S$ is the shape operator of $f_\xi$, i.e., that $K^h\circ D\xi=- Df_{\xi} \circ S$, and that $\xi_n\to\xi$ strongly in $W^{1,p}(\M;T\N)$.
We observe that
\[
\begin{split}
d\jmath\circ K^\h\circ D\xi_n  &=
d\jmath\circ(K^\h\circ D\xi_n + Df_{\xi_n}\circ S) \\
&\quad + (d\jmath\circ Df_\xi - d\jmath\circ Df_{\xi_n})\circ S \\
&\quad - d\jmath\circ Df_\xi \circ S 
\end{split}
\]
(the composition with $d\jmath$ is necessary in order to be able to add and subtract the various terms).
The first term on the right-hand side tends to zero in $L^p$ since the bending energy $\tE_p^B(\xi_n)$ tends to zero; the second term tends to zero in $L^p$ since $f_{\xi_n}\to f_\xi$ strongly in $W^{1,p}$.
Thus, the left-hand side converges strongly to $- d\jmath\circ Df \circ S$.

It follows from Lemma~\ref{lem:connector_convergence} that $K^\h\circ D\xi_n \to K^\h\circ D\xi$ in $L^p$ and that $K^h\circ D\xi=- Df \circ S$, thus proving that $S$ is the shape operator of $f_\xi$.
Since $\xi_n\wc \xi$ in $W^{1,p}$, $f_{\xi_n} \to f_{\xi}$ in $W^{1,p}$ and $K^\h\circ D\xi_n \to K^\h\circ D\xi$ in $L^p$, we obtain from \propref{prop:D1} that $\xi_n\to \xi$ strongly, as needed.
 
Lastly, $f_\xi$ satisfies the assumptions of Theorem~\ref{thm:reg} (which is proved in \secref{sec:smoothness} below), and thus $f_\xi$ is smooth, and by extension so is $\xi=\hn_{f_\xi}$.

\subsection{The case $p\le d$ via Sobolev truncation}

If $p \le d$, we first need to regularize $\xi_n \in W^{1,p}(\M;T\N)$, by replacing them with $\txi_n \in W^{1,\infty}(\M;T\N)$, which are close to $\xi_n$ in $W^{1,p}$ and also uniformly Lipschitz.
This is where the relaxation of the energy comes into play, since, even if $\xi_n = \hn_{f_{\xi_n}}$ project onto immersions $f_{\xi_n}\in \Imm_p(\M;\N)$, we do not know how to approximate them with more regular functions within the space $\Imm_p(\M;\N)$.
This construction follows a similar path as in other manifold-valued elasticity results, namely \cite{KMS19,KM21}, with some technical complications due to the structure of the energy in codimension-1  and the non-compactness of $T\N$.

First, we need the following pointwise estimate:
\begin{lemma}\label{lem:auxcalc}
Let $\xi\in W^{1,p}(\M;T\N)$.
Denote $M = \|S\|_\infty$, where $S$ is the shape operator in $\E_p$.
If at a point in $\M$, $|D\xi| \ge (3+2M)\sqrt{d+1}$, then at that point,
\beq
|D\xi| \le (3+2M)\brk{\dist(Df_\xi\oplus \xi,\SO(G,\h)) + |Df_\xi\circ S + K^\h\circ D\xi|}.
\label{eq:auxcalc}
\eeq
\end{lemma}

\begin{proof}
By the definition of the Sasaki metric,
\[
|D\xi| \le |Df_\xi| + |K^\h\circ D\xi|.
\]
Consider first the case where
$|Df_\xi| \ge 2\sqrt{d+1}$. 
Then, by the triangle inequality and the fact that all the elements in $\SO(G,\h)$ are of norm $\sqrt{d+1}$,
\beq\label{eq:dist_SO_bound}
\dist(Df_\xi\oplus \xi,\SO(G,\h)) \ge |Df_\xi| - \sqrt{d+1} \ge \sqrt{d+1}.
\eeq
Thus,
\beq\label{eq:Dxi_lower_bnd}
\begin{split}
|D\xi| 
&\le |Df_\xi| + |K^\h\circ D\xi| \\
&\le |Df_\xi| + |Df_\xi\circ S| + |Df_\xi\circ S  + K^\h\circ D\xi| \\
&\le (1+M)|Df_\xi| + |Df_\xi\circ S + K^\h\circ D\xi| \\
&\le (1+M)\dist(Df_\xi\oplus \xi,\SO(G,\h)) + (1+M)\sqrt{d+1} +  |Df_\xi\circ S + K^\h\circ D\xi| \\
&\le 2(1+M)\dist(Df_\xi\oplus \xi,\SO(G,\h)) +  |Df_\xi\circ S + K^\h\circ D\xi|,
\end{split}
\eeq
which implies \eqref{eq:auxcalc}.

Otherwise, if $|Df_\xi| < 2\sqrt{d+1}$ and $|D\xi| \ge (3+2M)\sqrt{d+1}$, then 
\[
\begin{split}
|D\xi| &\le |Df_\xi| + |K^\h\circ D\xi| \\
&\le  |Df_\xi| + |Df_\xi\circ S| + |Df_\xi\circ S  + K^\h\circ D\xi| \\
&\le (1+M)|Df_\xi| + |Df_\xi\circ S + K^\h\circ D\xi| \\
&\le 2(1+M)\sqrt{d+1} + |Df_\xi\circ S + K^\h\circ D\xi| \\
&\le \frac{2(1+M)}{3 + 2M}|D\xi| + |Df_\xi\circ S + K^\h\circ D\xi|,
\end{split}
\]
from which follows that
\[
|D\xi|  \le (3 + 2M) |Df_\xi\circ S + K^\h\circ D\xi|,
\]
which implies \eqref{eq:auxcalc}.
\end{proof}

For a given $\rho>2\sqrt{d+1}$, denote $T^\rho\N := \{v\in T\N ~:~ |v|\le \rho\}$, which is a compact submanifold of $T\N$.
Denote by $\vp:[0,\infty)\to[0,1]$ a continuous, non-negative, compactly-supported function satisfying $\vp(t)=1$ for $t\le 2\sqrt{d+1}$ and $\vp(t)=0$ for $t\ge \rho$, and let $\hat{\xi}_n = \vp(|\xi_n|)\xi_n$.
Note that $\hat{\xi}_n: \M \to T^\rho\N$ is uniformly bounded, and $|D\hat{\xi}_n| \le C|D\xi_n|$ for some $C$ depending only on $\vp$.
Furthermore, 
\[
\begin{split}
\{x\in \M ~:~ \hat{\xi}_n(x) \ne \xi_n(x)\} 
&\subset \{|\xi_n|>2\sqrt{d+1}\} \\
&\subset  \{\dist(Df_{\xi_n}\oplus \xi_n,\SO(G,\h)) > \sqrt{d+1}\},
\end{split}
\]
where the last inclusion follows from the same argument as in \eqref{eq:dist_SO_bound}.
Since $\tE_p(\xi_n)\to 0$, it follows that 
\[
\dist(Df_{\xi_n}\oplus \xi_n,\SO(G,\h)) \to 0 \qquad\text{in measure}, 
\]
hence the volumes of all the sets above tend to zero. 
On the one hand,
\[
\begin{split}
\int_\M |\xi_n - \hat{\xi}_n|^p\,\Volume &= \int_{\{\xi_n \ne \hat{\xi}_n\}} |\xi_n - \hat{\xi}_n|^p\,\Volume \\
&\lesssim  \int_{\{\xi_n \ne \hat{\xi}_n\}} |\xi_n|^p\,\Volume + \int_{\{\xi_n \ne \hat{\xi}_n\}} |\hat{\xi}_n|^p\,\Volume.
\end{split}
\]
The first term tends to zero as $n\to\infty$ because $|\xi_n|^p$ is equi-integrable (it is bounded in $W^{1,p}$); the second term tends to zero because $|\hat{\xi}|$ is uniformly bounded.  Moreover,
\[
\int_\M |d\iota \circ D\hat{\xi}_n - d\iota \circ D\xi_n|^p\, \Volume \lesssim
\int_{\{\dist(Df_{\xi_n}\oplus \xi_n,\SO(G,\h)) > \sqrt{d+1}\}} |D\xi_n|^p\, \Volume \lesssim \tE_p(\xi_n) \to 0,
\]
where the last inequality follows from \eqref{eq:Dxi_lower_bnd}. Thus,
\[
\|\iota \circ \hat{\xi}_n - \iota \circ \xi_n\|_{W^{1,p}} \to 0.
\]

Denote $u_n = \iota \circ \hat{\xi}_n \in W^{1,p}(\M;\R^D)$.
By \cite[Proposition~A.1]{FJM02b}, there exists a constant $C>0$, depending on $\M,\g,p$ and $\|S\|_\infty$, and there exists a sequence of Lipschitz maps $\bar{u}_n \in W^{1,\infty}(\M;\R^D)$ having uniform Lipschitz constant $C$ such that
\[
\Vol_\g\brk{\{x\in \M ~:~ \bar{u}(x) \ne u(x)\}} \le C\int_{\{|Du_n|>(3+2M)\sqrt{d+1}\}} |Du_n|^p\, \Volume,
\]
and
\[
\|\bar{u}_n - u_n\|_{W^{1,p}}^p \le C\int_{\{|Du_n|>(3+2M)\sqrt{d+1}\}} |Du_n|^p \Volume.
\]
Using Lemma~\ref{lem:auxcalc}, and the fact that $|Du_n| = |D\hat{\xi}_n|$ (since $\iota$ is an isometric immersion), we obtain that
\beq
\label{eq:estimates_tu_n1}
\Vol_\g\brk{\{x\in \M ~:~ \bar{u}(x) \ne u(x)\}} \le C'\tE_p(\hat{\xi}_n),
\eeq
and
\beq
\label{eq:estimates_tu_n2}
\|\bar{u}_n - u_n\|_{W^{1,p}}^p \le C'\tE_p(\hat{\xi}_n),
\eeq
for some $C'>0$.
Note that the image of $\bar{u}_n$ is not in $\iota(T\N)$, which means that we cannot identify $\bar{u}_n:\M\to\R^D$ with a function $\bar{\xi}_n:\M\to T\N$; to rectify this problem, we resort to a projection.
Since $\tE_p(\hat{\xi}_n)\to 0$, it follows from \eqref{eq:estimates_tu_n1} that for every $\e>0$, and $n\in \mathbb{N}$ large enough (depending on $\e$), every ball of radius $\e$ in $\M$ contains a point $x$ for which $\bar{u}_n(x) = u_n(x)\in\iota (T^\rho \N)$.
Since $\bar{u}_n$ are uniformly Lipschitz, we obtain that
\[
\max_{x\in \M} \dist\brk{\bar{u}_n(x), \iota (T^\rho \N)} \to 0.
\]
Thus, for $n$ large enough, $\bar{u}_n(x)$ lies in a tubular neighborhood of $\iota(T\N)$ on which the orthogonal projection operator $P$ is well-defined and smooth (this is where we use the compactness of $T^\rho\N$ and the reason we had to truncate $\xi_n$ to $\hat{\xi}_n$ --- we had to ensure that the image of $\bar{u}_n$ is uniformly close to a compact subset of $\iota (T\N)$).
Define $\tilde{u}_n = P\circ \bar{u}_n$, and $\txi_n = \iota^{-1}\circ \tilde{u}_n$.
A direct calculation, as in \cite[pp.~391--2]{KMS19}, shows that $\tilde{u}_n$ and $\txi_n$ are uniformly Lipschitz,  satisfy \eqref{eq:estimates_tu_n1} and \eqref{eq:estimates_tu_n2} , and furthermore, that $\tE_q(\txi_n)\to 0$ for any $q\in (1,\infty)$.

We can now apply the previous step for $\txi_n$ , for some $q>d\ge p$, and obtain that $\txi_n\to \xi=\hn_f$ in $W^{1,q}(\M;T\N)$ for some isometric immersion $f\in \Imm_\infty(\M;\N)$ whose shape operator is $S$. 
Since $\|\tilde{u}_n - u_n\|_{W^{1,p}}\to 0$, it follows by definition that $\dist_{W^{1,p}}(\txi_n,\hat{\xi}_n) \to 0$, and thus also $\dist_{W^{1,p}}(\txi_n,\xi_n) \to 0$. 
It therefore follows that $\xi_n \to \hn_f$ in $W^{1,p}(\M;T\N)$.

\section{Proof of Theorem~\ref{thm:reg}}
\label{sec:smoothness}

In this section we prove \thmref{thm:reg}.
First, we prove that $ f\in W^{2,2} $, in the sense that any coordinate representation of $f$ is in $W^{2,2}$. Then, in order to prove that $f$ is smooth, we prove an analog version of \cite[Lemma~3.1]{MS19}, where the range is a Riemannian manifold instead of $\R^n$ (for our purpose, we are only interested in codimension-1). This is achieved by the following two propositions:

\begin{proposition}
Let $f\in\Imm_p(\M;\N)$, $p> d$. Suppose that $f^*\h$ (which is a priori only defined almost-everywhere) is a smooth metric on $\M$.
Then $f\in W^{2,2}(\M;\N)$ (in coordinates).
\end{proposition}

In fact, the smoothness of $f^*\h$ is not needed here; $C^1$ would suffice, and possibly even Sobolev regularity.

\begin{proof}
Assume $d\geq 2$, hence $p>2$.
Let $\iota:(\N,\h) \to \R^D$ be an isometric embedding, i.e., $\iota(\N)$ is a smooth submanifold of $\R^D$ of codimension $D - d - 1$. 
We thicken $\M$ into $ \tM=\M\times (-h, h)^{D-d}$, with $h$ small enough, and endow it with the product metric 
\[
G= \begin{pmatrix}
f^*\h & 0\\
0 & I_{D-d}
\end{pmatrix}. 
\]
Let $\{\nu_i\}_{i=2}^{D-d}$ be a smooth orthonormal frame along the submanifold $\iota(\N)\subset \R^D$, orthogonally complementing the image of $d\iota$ on $\iota(\N)$.
We extend $f$ into a function  $F:\tM \mapsto \R^D$ by
\[ 
F(q,t_1,...,t_{D-d}) = \iota \circ f(q)+t_1 d\iota_{f(q)}(\hn_f(q)) +\sum_{i=2}^{D-d} t_i \nu_i(\iota\circ f(q)),
\] 
which we may also write in the form
\[ 
F = \iota \circ f \circ \pi +\pi_1 \cdot  (d\iota\circ \hn_f\circ \pi)  +\sum_{i=2}^{D-d} \pi_i \cdot  (\nu_i\circ \iota\circ f\circ \pi),
\] 
where $\pi:\tilde\M \to \M$ and $\pi_i :\tilde\M\to \R$ are the natural projections, $\pi(q,t_1,...,t_{D-d}) = q$ and $\pi_i(q,t_1,...,t_{D-d}) = t_i$, and in the second term we view $\hn_f$ as a map $\M\to T\N$.

Differentiating, $dF:T\tilde\M\to \R^D$ is given by
\[ 
\begin{split}
dF &=
d\iota \circ Df\circ D\pi + d\pi_1\otimes (d\iota\circ \hn_f\circ \pi) +\pi_1 \cdot (d^2\iota \circ D\hn_f\circ D\pi) \\
&+ \sum_{i=2}^{D-d} d\pi_i\otimes (\nu_i\circ \iota\circ f\circ \pi) 
+  \sum_{i=2}^{D-d} \pi_i \cdot (d\nu_i \circ D\iota \circ Df \circ D\pi).  
\end{split}
\] 
Consider
\[ 
A_f = d\iota \circ Df\circ D\pi + d\pi_1\otimes (d\iota\circ \hn_f\circ \pi) 
+ \sum_{i=2}^{D-d} d\pi_i\otimes (\nu_i\circ \iota\circ f\circ \pi).
\]
Since $f$ is an isometric immersion by the definition of $f^*\h$, $\iota$ is an isometric embedding and $\{\nu_i\}_{i=2}^{D-d}$ is an orthonormal frame along the submanifold $\iota(\N)\subset \R^D$, and $D\pi, d\pi_i$ are pairwise orthogonal by the choice of $G$, we obtain that $A_f  \in SO(G,\euc)$. 
Therefore, 
\[
\begin{split}
\dashint_{\tilde\M} \dist^2(dF,SO(G,\euc))\Vol_G
&\le \dashint_{\M} \dashint_{(-h_n, h_n)^{K-d}} |dF-A_f|^2\,dt\, \dVol_{f^*\h} \\
&\hspace{-4cm} \le C \dashint_{\M} \dashint_{(-h_n, h_n)^{D-d}} \brk{|t_1|^2|d^2\iota \circ D\hn_f\circ D\pi|^2 
 + \sum_{i=2}^{D-d} |t_i|^2 |d\nu_i \circ D\iota \circ Df \circ D\pi|^2} dt \,\dVol_{f^*\h}\\
&\hspace{-4cm}\le C \brk{\dashint_{(-h_n, h_n)^{D-d}}\sum_{i=1}^{D-d}|t_i|^2 dt  }\\
&\hspace{-4cm}\le Ch^2,
\end{split}
\]
where in the passage to the fourth line we used the facts that $d\hn_f$, and $df$ are bounded in $L^2(\M)$, and that $d\iota$,  $d^2\iota$ and  $d\nu_i$ are bounded in $ L^\infty $ as smooth functions on compact domains.

We have thus proved that $F$ satisfies the finite bending property defined in \cite[Eq.~(2.3)]{KS14}, where $F$ in the present  work corresponds to $f_h$ in \cite{KS14}. It follows from \cite[Thm.~5.1]{KS14} that $f$ (which unfortunately corresponds to $F$ in \cite{KS14}) is in $W^{2,2}$. 
\end{proof}

\begin{proposition}
Let $f\in\Imm_p(\M;\N)$, $p> d$. Suppose that $\g = f^*\h$ is a smooth metric on $\M$ and that the shape operator $S$ of $f(\M)$ in $\N$ is smooth. Then, $f$ is smooth.
\end{proposition}

The proof below works, with minor adjustments, also for a smooth second fundamental form instead of a smooth shape operator.

\begin{proof}
Since $f$ is continuous and since smoothness is a local property, we can consider the claim in coordinates; henceforth, $f:\Omega\subset\R^d \to \R^{d+1}$ and $\hn_f:\Omega\to\R^{d+1}$, satisfying the following set of equations,
\beq
\g_{ij}(x) = \h_{\alpha\beta}(f(x)) \, \partial_i f^\alpha(x) \, \partial_j f^\beta(x),
\label{eq:smoothness1}
\eeq
\beq
\partial_i \hn_f^\alpha(x) - \partial_i f^\beta(x)\, \Gamma^\alpha_{\beta\gamma}(f(x)) \hn_f^\gamma(x) = -\partial_j f^\alpha(x)\, S^j_i(x),
\label{eq:smoothness2}
\eeq
and
\beq
\h_{\alpha\beta}(f(x)) \, \partial_i f^\alpha(x) \, \hn_f^\beta(x) =0.
\label{eq:smoothness3}
\eeq
Here and below we use Latin indexes for coordinates in $\M$ and Greek indexes for coordinates in $\N$; the Einstein summation rule is assumed; $\Gamma^\alpha_{\beta\gamma}$ are the Christoffel symbols of the Levi-Civita connection in $\N$.
It is given that $\g_{ij}$ and $S^j_i$ are smooth.

It follows from the previous proposition that $f^\alpha\in W^{2,2}(\Omega)$, which together with \eqref{eq:smoothness1} implies that 
\[
\partial_i f^\alpha \in W^{1,2}(\Omega)\cap L^{\infty}(\Omega).
\]
It then follows from \eqref{eq:smoothness2} that
\[
\hn_f^\alpha\in W^{2,2}(\Omega)\cap W^{1,\infty}(\Omega).
\]

The rest of the proof follows a standard bootstrapping procedure, differentiating \eqref{eq:smoothness1} and \eqref{eq:smoothness3} and substituting \eqref{eq:smoothness2}, using the fact that $\{\partial_if(x)\}\cup\{\hn_f(x)\}$ forms an $f^*\h$-orthonormal basis for $\R^{d+1}$. The bootstrapping uses  the following product rule for Sobolev spaces \cite[Prop.~9.4]{Bre11}:
Let $u,v \in W^{1,p}(\Omega)\cap L^{\infty}(\Omega)$ with $1\leq p\leq \infty$. Then $uv\in W^{1,p}(\Omega)\cap L^{\infty}(\Omega)$  and 
\[
\partial_i (uv)=\partial_i u\, v+u\,\partial_i v.
\]
The fact that the application of the Leibniz rule requires to the very least functions in $W^{1,p}(\Omega)\cap L^{\infty}(\Omega)$ is the reason we had to first establish that $\partial_if^\alpha$ and $n_f^\alpha$ are in this space. 
\end{proof}

\paragraph{Acknowledgments}
RK was partially funded by ISF Grant 560/22 and CM was partially funded by ISF Grant 1269/19 and BSF Grant 2022076.

\appendix

\section{Quantitative stability in codimension-1}\label{sec:FJM_codim1}

In this appendix we prove a quantitative version of Theorem~\ref{thm:main} for the case $\N = \R^{d+1}$, under the assumption that the metric $\g$ and the shape operator $S$ are compatible with the geometry of $\R^{d+1}$.

\begin{theorem}
\label{thm:FJM_codim1}
Let $(\M,\g)$ be an oriented, connected, simply-connected, compact $d$-dimensional Riemannian manifolds with Lipschitz boundary, and let $p\in (1,\infty)$. 
Let $S$ by a smooth symmetric $(1,1)$ tensor field on $\M$. 
If $\g$ and $S$ satisfy the Gauss-Codazzi compatibility conditions, then there exists a constant $C$ depending on $(\M,\g)$, $S$ and $p$, such that there exists for every $f\in \Imm_p(\M;\R^{d+1})$ a smooth isometric immersion $f_0:\M\to \R^{d+1}$ having shape operator $S$, satisfying
\[
\begin{split}
& \|f - f_0\|_{W^{1,p}(\M;\R^{d+1})} + \|\hn_f - \hn_{f_0}\|_{W^{1,p}(\M;\R^{d+1})} \le \\
&\qquad C\brk{\|\dist(df,\O(\g,\euc))\|_{L^p(\M)} + \|\nabla \hn_f + df\circ S\|_{L^p(\M)}},
\end{split}
\]
where $\hn_f$ and $\hn_{f_0}$ are the unit normals to $f(\M)$ and $f_0(\M)$.
\end{theorem}

Note that the right-hand side in the equality is essentially $\E_p^{1/p}(f)$.
If $\M$ is not simply-connected, one can rather assume that $\g$ and $S$ are compatible in the sense that there exists an isometric immersion $\M\to \R^{d+1}$ whose shape operator is $S$.

\begin{proof}
Since $\g$ and $S$ are compatible, there exists, modulo a rigid transformation, a unique
smooth immersion $\iota:\M\to\R^{d+1}$, such that $\g = \iota^*\euc$ and $d\hn_\iota = -d\iota\circ S$, where  $\hn_\iota$ is the unit-normal to $\iota(\M)$. 

Consider the following diagram:
\[
\begin{xy}
(-10,0)*+{(\M,\g)}="S";
(30,0)*+{(\M_h,G)}="RS";
(30,20)*+{(\R^{d+1},\euc)}="R";
{\ar@{->}_{\zeta}"S";"RS"};
{\ar@{->}^{\iota}"S";"R"};
{\ar@{->}_{\Phi}"RS";"R"};
\end{xy}
\]
where $\M_h = \M\times(-h,h)$, with $h$ to be specified below, $\zeta(q) = (q,0)$, and 
\[
\Phi(q,t) = \iota(q) + t\, \hn_\iota(q).
\]
The metric $G$ on $\M_h$ is defined by
\[
\ip{(v,s),(w,r)}_{G_{(q,t}} = \ip{v - t\, S_q(v),w - t\, S_q(w)}_{\g_q} + sr ,
\]
where $v,w\in T_q\M$ and $s,r\in\R$. For $G$ to be a metric, $h$ has to be restricted; the choice of $h = 2/\|S\|_\infty$ guarantees that $h$ is metric. Furthermore, $G$ is equivalent to the product metric $\tilde{G} = \g + dt\otimes dt$, namely,
\[
c\brk{|v|_\g^2 + s^2} \le |(v,s)|_G^2 \le C\brk{|v|_\g^2 + s^2} 
\]
for some constants $c$ and $C$, depending only on $\g$ and $S$.
Finally, it follows that $\iota$, $\zeta$ and $\Phi$ are all isometric immersions.

Let $f\in \Imm_p(\M;\R^{d+1})$, which we extend into a map $F:\M_h\to\R^{d+1}$, given by
\[
F(q,t) = f(q) + t\, \hn_f(q),
\]
where $\hn_f:\M\to \R^{d+1}$ is the unit-normal to $f(\M)$.
Denote by $S_f$ the shape operator of $f(\M)$. 

Consider the map
\[
F\circ\Phi^{-1} : \Phi(\M_h) \subset\R^{d+1} \to \R^{d+1}.
\]
By the FJM inequality, there exists an isometry $Q\in\SO(d+1)$, such that
\[
\|d(F\circ\Phi^{-1}) - Q\|_{L^p(\Phi(\M_h);\R^{d+1})} \le C\,\|\dist(d(F\circ\Phi^{-1}),\SO(d+1))\|_{L^p(\Phi(\M_h);\R^{d+1})},
\]
where $C$ only depends on $\M$, $h$ and $p$. Since $\Phi$ is an isometric immersion, changing variables,
\beq
\|dF - Q\circ d\Phi\|_{L^p(\M_h;\R^{d+1})} \le C\,\|\dist(dF,\SO(G,\euc))\|_{L^p(\M_h;\R^{d+1})}.
\label{eq:FJMxxx1}
\eeq

The goal is to obtain an inequality involving only $f$.
We start with the right-hand side of \eqref{eq:FJMxxx1}.
First,
\[
dF = df \oplus \hn_f - t\, (df\circ S_f\oplus 0).
\]

Denote
\[
A = (\O(df) \oplus \hn_f) - t\, (\O(df)\circ S\oplus 0),
\]
where $\O(df)$ is the projection of $df$ on $O(\g,\euc)$ (for immersions the projection is unique).
We show that $A\in \SO(G,\euc)$. Indeed,
For $(q,t)\in\M_h$ and  $(v,s)\in T_{(q,t)}\M_h$,
\[
A(v,s) = \O(df)(v) + s\, \hn_f - t\, O(df)\circ S(v).
\]
Since $\O(df)(v) \perp \hn_f$ and since $\O(df)\in \O(\g,\euc)$,
\[
|A(v,s)|^2 = |\O(df)(v)  - t\, O(df)\circ S(v)|^2 + s^2 = |v - t\,S(v)|^2 + s^2 = |(v,s)|^2.
\]
The condition on the orientation holds by the very definition of $\hn_f$. 

Thus,
\[
\begin{split}
\dist(dF,\SO(G,\euc)) &\le |dF - A| \\
&\le |(df - \O(df)) \oplus \hn_f - t\, ((df\circ S_f - \O(df)\circ S)\oplus 0)| \\
&\le |(df - \O(df)) \oplus \hn_f| + t\, |(df\circ S_f - \O(df)\circ S)\oplus 0|,
\end{split}
\]
where all the norms are with respect to $G$ and $\euc$. By the equivalence between $G$ and the product metric $\tilde{G} = \g + dt\otimes dt$, the right hand side can be bounded by
\[
C  \brk{|df - \O(df)| + |df\circ S_f - \O(df)\circ S|},
\]
where $C$ depends only on $\g$, $S$ and $h$. 
Using once again the equivalence of the metrics on $\M_h$,
\beq
\|\dist(dF,\SO(G,\euc))\|_{L^p(\M_h)} \le C\brk{\|\dist(df,\O(\g,\euc))\|_{L^p(\M)} + \|df\circ S_f - \O(df)\circ S\|_{L^p(\M)}}.
\label{eq:FJMxxx2}
\eeq
for some constant $C$ depending only on $(\M,\g)$ and $p$.

We proceed with the left-hand side of \eqref{eq:FJMxxx1}.
We have
\[
\begin{split}
Q\circ d\Phi &= Q\circ(d\iota \oplus \hn_\iota) - t\, Q\circ (d\iota\circ S\oplus 0) \\
&=d(Q\circ \iota) \oplus \hn_{Q\circ \iota} - t\, (d(Q\circ\iota)\circ S\oplus 0),
\end{split}
\]
hence
\[
\begin{split}
dF - Q\circ d\Phi 
	&= df \oplus \hn_f - df_0 \oplus \hn_{f_0} - t\, \brk{(df\circ S_f  - df_0 \circ S_{f_0})\oplus 0}\\
	&= df \oplus \hn_f - df_0 \oplus \hn_{f_0} - t\, \brk{(d\hn_f  - d\hn_{f_0})\oplus 0},
\end{split}
\]
where $f_0 = Q\circ \iota$, which is an isometric immersion of $(\M,\g)$ having shape operator $S_{f_0} = S$. 

Therefore,
\[
\begin{split}
|dF - Q\circ d\Phi|^2 &\gtrsim |df- df_0|^2 + |\hn_f -  \hn_{f_0}|^2 
+ t^2|d\hn_f  - d\hn_{f_0}|^2  \\
&\quad + 2t (df- df_0, d\hn_f  - d\hn_{f_0})_\g. \\
\end{split}
\]
Define the set 
\[
A_+=\{q\in \M ~:~ ((df)_q- (df_0)_q, (d\hn_f)_q  - (d\hn_{f_0})_q)_\g \ge 0 \} \subset \M
\]
and let
\[
D = \{(q,t)\in \M_\h ~:~ q\in A_+\,\,\, t\ge 0\} 
\cup  \{(q,t)\in \M_\h ~:~ q\not\in A_+ \,\,\, t\le 0\}  \subset \M_h.
\]
On the set $D$,
\[
|dF - Q\circ d\Phi|^2 \gtrsim |df- df_0|^2 + |\hn_f -  \hn_{f_0}|^2 + t^2|d\hn_f  - d\hn_{f_0}|^2,
\]
and thus
\[
\begin{split}
\int_{\M_h} |dF - Q\circ d\Phi|^p \,\dVol_G 
&\ge \int_{D} |dF - Q\circ d\Phi|^p \,\dVol_G \\
&\gtrsim  \int_D \brk{|df- df_0|^p + |\hn_f -  \hn_{f_0}|^p + |t|^p|d\hn_f  - d\hn_{f_0}|^p} \,\dVol_G \\
&\gtrsim  \int_D \brk{|df- df_0|^p + |\hn_f -  \hn_{f_0}|^p + |t|^p|d\hn_f  - d\hn_{f_0}|^p} \,\dVol_{\tilde{G}} \\
&= \int_\M \brk{h|df- df_0|^p + h|\hn_f -  \hn_{f_0}|^p + \frac{h^{p+1}}{p+1}|d\hn_f  - d\hn_{f_0}|^2} \,\dVol_{\g},
\end{split}
\]
where in the passage to the third line we used the equivalence of $G$ and the product metric $\tilde{G}$, and in the passage to the last line we used Fubini's theorem.
From this and the Poincar\'e inequality (possibly by translating $f_0$) we obtain
\beq\label{eq:FJMxxx3}
\|dF - Q\circ d\Phi\|_{L^p(\M_h;\R^{d+1})} \ge c\brk{ \|f - f_0\|_{W^{1,p}(\M;\R^{d+1})} + \|\hn_f - \hn_{f_0}\|_{W^{1,p}(\M;\R^{d+1})} },
\eeq
where $c>0$ depends only on $(\M,\g)$, $p$ and $h$.
Combining \eqref{eq:FJMxxx1}, \eqref{eq:FJMxxx2} and \eqref{eq:FJMxxx3}, we obtain the desired result.
\end{proof}
 
\footnotesize{

\providecommand{\bysame}{\leavevmode\hbox to3em{\hrulefill}\thinspace}
\providecommand{\MR}{\relax\ifhmode\unskip\space\fi MR }
\providecommand{\MRhref}[2]{%
  \href{http://www.ams.org/mathscinet-getitem?mr=#1}{#2}
}
\providecommand{\href}[2]{#2}

}

\end{document}